\documentclass[1p]{elsarticle}

\usepackage{algorithm}
\usepackage{algorithmic}
\usepackage{amsmath}
\usepackage{amsfonts}
\usepackage{amssymb}
\usepackage{url}
\usepackage{bm}
\usepackage{paralist}
\usepackage{lineno}
\usepackage{latexsym}
\usepackage[latin1]{inputenc}
\usepackage{amsthm}
\newtheorem{theorem}{Theorem}
\newtheorem{lemma}[theorem]{Lemma}
\newdefinition{rem}{Remark}
\newtheorem{proposition}[theorem]{Proposition}

\journal{Linear Algebra and its Applications%---Special Issue on Tensors and
}

\DeclareMathOperator{\diag}{diag}  %
\DeclareMathOperator{\rank}{rank} %
\DeclareMathOperator{\spn}{span}   %
\DeclareMathOperator{\range}{Range}   %

\newcommand{\cA}{\mathcal{A}}
\newcommand{\cB}{\mathcal{B}}
\newcommand{\cC}{\mathcal{C}}
\newcommand{\cD}{\mathcal{D}}

\newcommand{\cF}{\mathcal{F}}

\newcommand{\cH}{\mathcal{H}}

\newcommand{\cK}{\mathcal{K}}
\newcommand{\cS}{\mathcal{S}}

\newcommand{\RR}{\mathbb{R}}

\renewcommand{\a}{\alpha}
\renewcommand{\b}{\beta}
\newcommand{\g}{\gamma}

\newcommand{\s}{\sigma}

\newcommand{\<}{\langle}
\renewcommand{\>}{\rangle}

\newcommand{\x}{\times}

\newcommand{\wh}[1]{\widehat{ #1 }}

\newcommand{\tp}{^{\sf T}}

\newcommand{\tml}[3][]{\bm{\left(} #2 \bm{\right)}_{ #1} \bm{\cdot} #3}
\newcommand{\tmr}[3][]{#2 \bm{\cdot} \bm{\left(} #3 \bm{\right)}_{ #1}}
\newcommand{\ctp}[2][]{\left\< #2 \right\>_{#1}}

\begin{document}

\begin{frontmatter}



\title{Krylov-Type  Methods for Tensor Computations\tnoteref{t1}}
\tnotetext[t1]{This work was
    supported by the Swedish Research Council.}
% \tnotetext[t2]{The second title footnote which is  
%    longer than the first one and with an intention to fill
%    in up more than one line while formatting.} 


\author[lars,berkant]{Berkant Savas}
\ead{berkant@cs.utexas.edu}

\author[lars]{Lars Eld\'{e}n}
\ead{laeld@math.liu.se}

\address[lars]{Department of Mathematics,
Link\"{o}ping University,
SE-581 83 Link\"{o}ping,
Sweden}
\address[berkant]{Institute for Comp. Engin. and Sciences,
University of Texas at Austin,
Austin, TX 78712,~USA}

% \title{This is a specimen title\tnoteref{t1,t2}}
% \author[rvt]{C.V.~Radhakrishnan\corref{cor1}\fnref{fn1}}
% \ead{cvr@river-valley.com}
 
% \author[rvt,focal]{K.~Bazargan\fnref{fn2}}
% \ead{kaveh@river-valley.com}
 
% \author[els]{S.~Pepping\corref{cor2}\fnref{fn1,fn3}}
% \ead[url]{http://www.elsevier.com}
 
% \cortext[cor1]{Corresponding author}
% \cortext[cor2]{Principal corresponding author}
% \fntext[fn1]{This is the specimen author footnote.}
% \fntext[fn2]{Another author footnote, but a little more longer.}
% \fntext[fn3]{Yet another author footnote. Indeed, you can have
%    any number of author footnotes.}

\begin{abstract}
  Several Krylov-type procedures are introduced that generalize matrix
  Krylov methods for tensor computations. They are denoted \textit{minimal
    Krylov recursion}, \textit{maximal Krylov recursion},
  \textit{contracted tensor product Krylov recursion}. It is proved that
  the for a given tensor with low rank, the minimal Krylov recursion
  extracts the correct subspaces associated to the tensor within certain
  number of iterations.  An optimized minimal Krylov procedure is described
  that gives a better tensor approximation for a given multilinear rank
  than the standard minimal recursion. The maximal Krylov recursion
  naturally admits a Krylov factorization of the tensor.  The tensor Krylov
  methods are intended for the computation of low-rank approximations of
  large and sparse tensors, but they are also useful for certain dense and
  structured tensors for computing their higher order singular value
  decompositions or obtaining starting points for the best low-rank
  computations of tensors.  A set of numerical experiments, using real life
  and synthetic data sets, illustrate some of the properties of the tensor
  Krylov methods.
\end{abstract}

\begin{keyword}
Tensor \sep Krylov-type method \sep tensor approximation \sep Tucker model \sep multilinear algebra \sep multilinear rank \sep sparse tensor \sep information science

% \PACS
\emph{AMS:} 15A69 \sep 65F99 %, 41A46, 65D15
\end{keyword}

\end{frontmatter}


%% main text ====================START=====================================
%\linenumbers

\section{Introduction}
\label{sec:intro}
Large-scale problems in engineering and science often require solution
of sparse linear algebra problems, such as systems of equations, and
eigenvalue problems.  Recently
\cite{bhk:06,bader07b,chew07,kolda06b,kolda05,kolda09}
it has been shown that several applications in information sciences,
such as web link analysis, social networks, and cross-language
information retrieval, generate large data sets that are sparse
tensors. In this paper we introduce new methods for efficient
computations with large and sparse tensors.

Since the 1950's Krylov subspace methods have been developed so that
they are now one of the  main classes of algorithms for solving
iteratively large and sparse matrix problems.  Given a square matrix
$A \in \RR^{n \x n}$ and a starting vector $u \in \RR^n$ the
corresponding $k$-dimensional Krylov subspace  is 
\begin{equation*}\label{eq:krylovSubsp}
\cK_k(A,u) = \spn \{u, Au, A^2 u, \dots, A^{k-1}u \}.
\end{equation*}
In floating point arithmetic the vectors in the Krylov subspace are
useless unless they are orthonormalized.  Applying Gram-Schmidt
orthogonalization one obtains the Arnoldi process, which
 generates an orthonormal basis for the Krylov subspace
$\cK_k(A,u)$. In addition, the Arnoldi process generates the
factorization
\begin{equation}
\label{eq:arn-fact}
A U_k = U_{k+1} H_k,
\end{equation}
where $U_k = [u_1\, \dots \, u_k]$, and $U_{k+1} = [U_k\,\, u_{k+1}]$
with orthonormal columns, and $H_k$ is a Hessenberg matrix with the
orthonormalization coefficients. Based on the factorization
\eqref{eq:arn-fact} one can compute an approximation of the solution
of a linear system or an eigenvalue problem by projecting onto the
space spanned by the columns of $U_k$, where $k$ is much smaller than
the dimension of $A$; on that subspace the operator $A$ is represented
by the small matrix $H_k$. This approach is particularly useful for
large, and sparse problems, since it uses the matrix $A$ in
matrix-vector multiplications only.

Projection to a low-dimensional subspace is a common technique in many
areas of information science. This is also the case in applications
involving tensors. One of the main theoretical and algorithmic problems
researchers have addressed is the computation of low  rank
approximation of a given tensor
\cite{elden09,savaslim10,ishteva09b,oseledets08,khor09}.  The two main
approaches are the Canonical Decomposition \cite{cach70,harsh70} and the
Tucker decomposition \cite{tucke64}; we are concerned with the latter. 

The following question arises naturally:
\begin{quote}
 \emph{Can Krylov methods be generalized to tensors, to be used for the 
 projection to  low-dimensional subspaces? }
\end{quote}
We answer this question in the affirmative, and describe several
alternative ways one can generalize Krylov subspace methods for
tensors. Our method is inspired by Golub-Kahan bidiagonalization
\cite{golkah:65}, and the Arnoldi method, see e.g. \cite[p.
303]{stew:01}. In the bidiagonalization method two sequences of
orthogonal vectors are generated; for a tensor of order three, our
procedures generates three sequences of orthogonal vectors.  Unlike
the bidiagonalization procedure, it is necessary to perform Arnoldi
style orthogonalization of the generated vectors explicitly. For
matrices, once an initial vector has been selected, the whole sequence
is determined uniquely. For tensors, there are many ways in which the
vectors can be generated. We will describe three principally different
tensor Krylov methods. These are the \textit{minimal Krylov
  recursion}, \textit{maximal Krylov recursion} and \textit{contracted
  tensor product Krylov recursion}. In addition we will discuss the
implementation of an optimized version \cite{gos10} of the minimal
Krylov recursion, and we will show how to deal with tensors that are
small in one mode. For a given tensor $\cA$ with $\rank(\cA)= (p,q,r)$
the minimal Krylov recursion can extract the correct subspaces
associated to $\cA$ in $\max\{p,q,r\}$ iterations. The maximal Krylov
recursion admits a tensor Krylov factorization that generalizes the
matrix Krylov factorization.  The contracted tensor product Krylov
recursion is a generalization of  the matrix Lanczos method applied to
symmetric matrices $A\tp A$ and $A A\tp$.

%The aim of this paper is to present some concepts that we believe
%are new and potentially important, and to perform a preliminary exploration of those concepts. 
Although our main motivation is to develop efficient methods for large
and sparse tensors, the methods are useful for other tasks as well. In
particular, they can be used for obtaining starting points for the
\textit{best} low rank tensor approximation problem, and for tensors
with relatively low multilinear rank they provide a way of speeding up
the computation of the Higher Order SVD (HOSVD) \cite{latha00}. The
latter part is done by first computing a full factorization using the
minimal Krylov procedure and then computing the HOSVD of the much
smaller core tensor that results from the approximation.

The paper is organized as follows. The necessary tensor concepts are
introduced in Section \ref{sec:concepts}. The Arnoldi and Golub-Kahan
procedures are sketched in Section \ref{sec:two-matrix}. In Section
\ref{sec:tensor-krylov} we describe different variants of Krylov methods
for tensors.  Section \ref{sec:numerical} contains  numerical examples
illustrating various aspects of the proposed methods. 

As this paper is a first introduction to Krylov methods for tensors, we do
not imply that it gives a comprehensive treatment of the subject. Rather
our aim is to outline our discoveries so far, and point to the similarities
and differences between the tensor and matrix cases.

\section{Tensor Concepts}
\label{sec:concepts}

\subsection{Notation and Preliminaries}
Tensors will be denoted by calligraphic letters, e.g
$\mathcal{A},\mathcal{B}$, matrices by capital roman letters and
vectors by lower case roman letters.  In order not to burden the
presentation with too much detail, we sometimes will not explicitly
mention the dimensions of matrices and tensors, and assume that they
are such that the operations are well-defined. The whole presentation
will be in terms of tensors of order three, or equivalently 3-tensors.
The generalization to order-$N$ tensors is obvious.

  We
will use the term tensor in a restricted sense, i.e. as a
3-dimensional array of real numbers, $\cA \in \RR^{l \times m \times n}$,
where the vector space is  equipped with some algebraic
structures to be defined. The
different ``dimensions'' of the tensor are referred to as
\emph{modes}.   We will use
both standard subscripts  and ``MATLAB-like'' notation: a particular
tensor element will be denoted in two equivalent ways,
\[
\cA(i,j,k) = a_{ijk}.
\]
%We will refer to subtensors in the following way.
A subtensor obtained by fixing one of the indices is called a
\emph{slice}, e.g.,
\[
\cA(i,:,:).
\]
Such a slice can be  considered as an order-3 tensor, but also as a
matrix. 

A \emph{fibre} is a subtensor, where all indices but one are fixed,
\[
\cA(i,:,k).
\]
For a given third order tensor, there are three associated subspaces, one for each mode. These subspaces are given by 
\begin{align*}
& \range \{ \cA(:,j,k) \; | \; j = 1:m, \; k = 1:n  \},\\
& \range \{ \cA(i,:,k) \; | \; i = 1:l, \; k = 1:n  \},\\
& \range \{ \cA(i,j,:) \; | \; i = 1:l, \; j = 1:m  \}.
\end{align*}
The \emph{multilinear rank} \cite{hit:27,desilva08} of the tensor is said
to be equal to $(p,q,r)$ 
if the dimension of these subspaces are $p,$ $q,$ and $r$, respectively. 

It is customary in numerical linear algebra to write out column
vectors with the elements arranged vertically, and row vectors with
the elements horizontally. This becomes inconvenient when we are
dealing with more than two modes. Therefore we will not make a
notational distinction between mode-1, mode-2, and mode-3 vectors, and
we will allow ourselves to write all vectors organized vertically. It
will be clear from the context  to which mode the vectors belong.
However, when dealing with matrices, we will often talk of 
them as consisting of column vectors.

\subsection{Tensor-Matrix Multiplication}
\label{sec:ten-matmult}

We define \emph{multilinear multiplication of a tensor by a matrix} as
follows.  For concreteness we first present multiplication by one matrix
along the first mode and later for all three modes simultaneously. The
mode-$1$ product of 
a tensor $\cA \in \RR^{l \times m \times n}$ by a matrix $U \in \RR^{p
  \times l}$ is defined\footnote{The notation \eqref{eq:contra}-\eqref{eq:mat-tensor} was
suggested by Lim \cite{desilva08}. An alternative notation was
  earlier given in \cite{latha00}.  Our $\tml[d]{X}{\cA}$ is the same
  as $\cA \times_d X$ in that system.} 
\begin{equation}\label{eq:contra}
 \RR^{p \times m \times n} \ni \mathcal{B} =  \tml[1]{U}{ \cA} ,\qquad
b_{ijk} = \sum_{\a=1}^{l} u_{i \a} a_{ \a jk}  .
\end{equation}
This means that all mode-$1$ fibres in the $3$-tensor $\cA$ are
multiplied by the matrix $U$. Similarly, mode-$2$ multiplication by a
matrix $V \in \RR^{q \x m}$ means that all mode-$2$ fibres are
multiplied by the matrix $V$.  Mode-$3$ multiplication is analogous.
With a third matrix $W\in \RR^{r \x n}$, the tensor-matrix
multiplication\footnote{To clarify the presentation, when dealing with
  a general third order tensor $\cA$ , we will use the convention that
  matrices or vectors $U,U_{k}, u_{i}$, $V, V_{k}, v_{i}$ and
  $W,W_{k},w_{i}$ are exclusively multiplied along the first, second,
  and third mode of $\cA$, respectively, and similarly with matrices
  and vectors $X,Y,Z, x, y, z$.} in all modes is given by
\begin{equation}
  \label{eq:mat-tensor}
\RR^{p \x q \x r}	\ni \cB = \tml{U,V,W}{\cA}, \qquad b_{ijk} = \sum_{\a,\b,\g=1}^{l,m,n} u_{i \a} v_{j \b} w_{k \g}a_{ \a \b \g},
\end{equation}
where the mode of each multiplication is understood from the order in
which the matrices are given.

It is convenient to introduce a separate notation for multiplication
by a transposed matrix $\bar{U} \in \RR^{l \times p}$:
\begin{equation}
  \label{eq:cov}
\RR^{p \times m \times n} \ni \mathcal{C}= \tml[1]{\bar{U}\tp }{\cA}
=\tmr[1]{\cA}{\bar{U}}, \qquad
c_{ijk} = \sum_{\a=1}^{l} a_{\a jk} \bar{u}_{\a i}.
\end{equation}
Let $u \in \RR^l $ be a  vector and $\cA \in \RR^{l \times m
  \times n}$ a tensor. Then
\begin{equation}
  \label{eq:ident-1}
 \RR^{1 \times m   \times n} \ni \cB := \tml[1]{u\tp}{\cA} =
\tmr[1]{\cA}{u} \equiv B \in \RR^{m   \times  n}.
\end{equation}
Thus we identify a tensor with a singleton dimension with  a
matrix. Similarly, with $u \in \RR^l$ and $w \in \RR^n$, we will
identify
\begin{equation}
  \label{eq:ident-13}
\RR^{1 \times m \times 1} \ni \cC :=   \tmr[1,3]{\cA}{u,w}  \equiv
c \in \RR^m,
\end{equation}
i.e., a tensor of order three with two singleton dimensions is
identified with a vector, here in the second mode. Since formulas like \eqref{eq:ident-13} have key importance in this paper, we will state the other two versions as well;
\begin{equation}
  \label{eq:ident-12-23}
\tmr[1,2]{\cA}{u,v}  \in \RR^{n},\qquad \tmr[2,3]{\cA}{v,w}  \in \RR^{l},
\end{equation}
where $v \in \RR^{m}$.

\subsection{Inner Product, Norm, and Contractions}
\label{sec:norm}

Given two tensors $\cA$ and $\cB$ of the same dimensions, we define
the \emph{inner product},
\begin{equation}\label{eq:inner-prod}
\< \cA , \cB \> = \sum_{\a,\b,\g} a_{\a \b \g} b_{\a \b \g}.
\end{equation}
The  corresponding \emph{tensor norm} is
\begin{equation}
  \label{eq:norm}
  \| \cA \| = \< \cA, \cA \>^{1/2}.
\end{equation}
This \emph{Frobenius norm} will be used throughout the paper.
As in the matrix case, the norm is invariant under orthogonal
transformations, i.e.
\[
  \| \cA \| = \left\| \tml{U,V,W}{\cA} \right\| = \| \tmr{\cA}{P,Q,S} \|,
\]
for orthogonal matrices $U$, $V$, $W$, $P$, $Q$, and $S$. This is obvious 
from the fact that multilinear multiplication by orthogonal matrices
does not change the Euclidean length of the corresponding fibres of the tensor.

For convenience we will denote the inner product of
  vectors $x$ and $y$ in any mode (but, of course, the same) by $x\tp
  y$. Let $v = \tmr[1,3]{\cA}{u,w}$; then, for a matrix $V=[v_1 \; v_2
  \; \cdots \; v_p]$ of mode-2 vectors, we have
\[
V\tp v = \tml[2]{V\tp}{(\tmr[1,3]{\cA}{u,w})} = \tmr{\cA}{u,V,w} \in
\RR^{1 \times p \times 1}. 
\]

The following well-known result will be needed.

\begin{lemma}\label{lem:LS}
  Let $\cA \in \RR^{l \times m \times n}$ be given along with three
  matrices with orthonormal columns, $U \in \RR^{l \times p}$, $V \in
  \RR^{m \times q}$, and $W \in \RR^{n \times r}$, where $p \leq l$,
  $q \leq m$, and $r \leq n$. Then the least squares problem
\[
\min_{\cS} \| \cA - \tml{U,V,W}{\cS} \|
\]
has the unique solution
\[
\cS= \tml{U\tp ,V\tp ,W\tp }{\cA} = \tmr{\cA}{U,V,W}.
\]
The elements of the tensor $\cS$ are given by 
\begin{equation}
  \label{eq:elem-S}
  s_{\lambda\mu\nu}=\tmr{\cA}{u_\lambda,v_\mu,w_\nu}, \quad 1 \leq
  \lambda \leq p, \quad 1 \leq \mu \leq q, \quad 1 \leq \nu \leq r.
\end{equation}
\end{lemma}
\begin{proof}
  The proof is a straightforward generalization of the corresponding
  proof for matrices. Enlarge each of the matrices so that it becomes
  square and orthogonal, i.e., put
\[
\bar U = [U \,\, U_\perp], \quad \bar V = [V \, \,V_\perp], \quad
\bar W = [W \,\, W_\perp].
\]
Introducing the residual $ \mathcal{R} =\cA - \tml{U,V,W}{\cS}$ and using the invariance of the norm under orthogonal
transformations, we get
\begin{equation*}
\| \mathcal{R}\|^2 =
\left\|\tmr{\mathcal{R}}{\bar{U},\bar{V},\bar{W}}
 \right\|^2  =
\| \tmr{\cA}{U ,V ,W} - \cS
\|^2 + C^2,
\end{equation*}
where $C^2 = \| \tmr{\cA}{U_\perp,V_\perp,W_\perp} \|^2$ does not
depend on $\cS$. The relation \eqref{eq:elem-S} is obvious from the
definition of tensor-matrix product.
\end{proof}

The inner product \eqref{eq:inner-prod} can be considered as a special
case of the \emph{contracted product of two tensors},
cf. \cite[Chapter 2]{kono:63},
which is a tensor (outer) product followed by a contraction along specified
modes. Thus, if $\cA$ and $\cB$ are $3$-tensors,
we define, using essentially the notation of \cite{bader06b},

\addtocounter{equation}{1}
\begin{align}
  \cC &= \left\< \cA, \cB \right\>_{1}\,, &
   c_{jkj'k'} &= \sum_\a a_{\a j k} b_{\a j' k'}\,, & &
   \mbox{($4$-tensor)}\,, \label{eq:cntr1}\tag{\theequation.a}\\
  D &= \left\< \cA, \cB \right\>_{1,2}\,, &
   d_{k k'} & = \sum_{\a,\b} a_{\a \b k} b_{\a \b k'}\,, & &
   \mbox{($2$-tensor)}, \label{eq:cntr2}\tag{\theequation.b}\\
  e & = \< \cA, \cB \> = \left\< \cA, \cB \right\>_{1\dots3}\,, & e &
   = \sum_{\a,\b, \g} a_{\a \b \g} b_{\a \b \g}\,, & &
   \mbox{(scalar)}. \label{eq:cntr3}\tag{\theequation.c}
\end{align}
It is required that the contracted dimensions are equal in the two
tensors. We will refer to the first two as \emph{partial
  contractions}. The subscript 1 in $\langle\mathcal{A},\mathcal{B} \rangle_{1}$ and 1,2 in $\langle\mathcal{A},\mathcal{B} \rangle_{1,2}$ indicate that the contraction is over the first mode in both arguments and in the first and second mode in both arguments, respectively. It is also convenient to introduce a notation when contraction is performed in all but one mode. For example the product in \eqref{eq:cntr2} may also be written 
\begin{equation}
%\left\< \cA, \cB \right\>_{1} \equiv \left\< \cA, \cB \right\>_{-(2,3)}\,, \qquad
\left\< \cA, \cB \right\>_{1,2} \equiv \left\< \cA, \cB \right\>_{-3}\,.
\end{equation}
The definition of contracted products  is valid also  when the 
tensors are of different order. The only assumption is that the dimension
of the correspondingly contracted  modes are the same in the two
arguments. The dimensions of the resulting product are in the order
given by the non-contracted modes of the first argument followed by
the non-contracted modes of the second argument.

%
%\subsection{Matricization}
%\label{sec:matricizaion}
%
%A tensor can be \emph{matricized}\footnote{Alternatively,  \emph{unfolded} \cite{lmv:00a} or \emph{flattened}
%\cite{eccv:02}.} in many different ways. We use the convention
%introduced in \cite{} (which differs from that in
%\cite{bako:06a,bako:08}).
%Let $\textsf{r}= [r_1, \cdots, r_L]$ be the modes of $\cA$ mapped to
%the rows and $\textsf{c} = [c_1, \cdots, c_M]$ be the modes of $\cA$
%mapped to the columns. The matricization is denoted
%%
%\begin{equation}
%\label{eq:mat-ten-size}
%A^{(\textsf{r};\textsf{c})} \in R^{J \times K}, \quad \text{ where } \quad
%J = \prod_{i=1}^L I_{r_i}, \quad \text{ and } \quad K = \prod_{i=1}^M I_{c_i}.
%\end{equation}
%%
%For a given order-$N$ tensor $\cA$, the element $\cA(i_1, \dots, i_N)$ is
%mapped to
%$A^{(\textsf{r};\textsf{c})}(j,k)$ where
%%
%\begin{align}
%j &= 1 + \sum_{l=1}^L \left[ \left( i_{r_{L-l+1}} - 1\right)
%\prod_{l'=1}^{l-1} I_{r_{L-l'+1}} \right], \label{eq:row-coord} \\
%k &= 1 + \sum_{m=1}^M \left[ \left( i_{c_{M-m+1}} - 1\right)
%\prod_{m'=1}^{m-1} I_{c_{M-m'+1}} \right]. \label{eq:col-coord}
%\end{align}
%
%
%
%

\subsection{Tensor Matricization} 
Later on we will also need the notion of tensor matricization. Any
given third order tensor $\cA\in \RR^{l \x m \x n}$ can be matricized
alone its different modes. These matricizations will be written as
$A^{(1)}$, which is an $l \x mn$ matrix,  $A^{(2)}$, which is an $m \x
ln$ matrix, and $A^{(3)}$ is an $n \x lm$ matrix. The exact relations
of the entries of $\cA$ to the three different matricizations can be
found in \cite{elden09}. It is sufficient, for our needs in this
paper, to recall that the matricizations of a given multilinear
tensor-matrix product $\cB = \tml{U,V,W}{\cA}$ have the following
forms: 
\begin{align*}
B^{(1)} & = U A^{(1)}(V \otimes W)\tp, \\
B^{(2)} & = V A^{(2)}(U \otimes W)\tp, \\
B^{(3)} & = W A^{(3)}(U \otimes V)\tp.
\end{align*}

\section{Two Krylov Methods for Matrices}
\label{sec:two-matrix}

In this section we will describe briefly the two matrix Krylov methods
that are the starting point of our generalization to tensor Krylov
methods.

\subsection{The Arnoldi Procedure}
\label{sec:arnoldi}

The Arnoldi procedure is used to compute a low-rank
approximation/factor\-iza\-tion \eqref{eq:arn-fact} of a
square, in general nonsymmetric matrix $A$. It requires a starting
vector $u_1 =: U_{1}$ (or, alternatively, $v_1=:V_1$), and in each step the new vector is orthogonalized
against all previous vectors using the modified  Gram-Schmidt process.  We
present the Arnoldi procedure in the style of \cite[p. 303]{stew:01}.

\begin{algorithm}
\caption{Arnoldi Procedure}
\label{alg:arnoldi}
\begin{algorithmic}
\FOR {$i=1,2,\ldots,k$}
\STATE \textbf{1 }$h_i = U_i\tp A u_i$
\STATE \textbf{2 }$h_{i+1,i} u_{i+1}=A u_i - U_i h_i$
\STATE \textbf{3 }$U_{i+1} = [U_{i}\; u_{i+1}]$
\STATE \textbf{4 }$H_i = \begin{bmatrix}H_{i-1} & h_i \\
                                        0 & h_{i+1,i}
                         \end{bmatrix}
       $
\ENDFOR
\end{algorithmic}
\end{algorithm}

The constant $h_{i+1,i}$ is used to normalize the new vector to length one.
Note that the matrix $H_k$ in the the factorization
\eqref{eq:arn-fact} is obtained by collecting the orthonormalization
coefficients $h_i$ and $h_{i+1,i}$ in an upper Hessenberg matrix.

\subsection{Golub-Kahan Bidiagonalization}
\label{sec:GKB}

Let  $A \in \RR^{m \times n}$ be a matrix, and let
$\b_1 u_1, v_0 = 0$, where $\|u_1\| = 1$, be  starting vectors.   The
Golub-Kahan bidiagonalization procedure \cite{golkah:65} is defined by
the following recursion.
\begin{algorithm}
\caption{Golub-Kahan bidiagonalization}
\label{alg:gk-bidiag}
\begin{algorithmic}
\FOR{$i = 1,2,\dots,k$}
\STATE \textbf{1 } $\a_i v_i = A\tp u_i - \b_i v_{i-1}$
\STATE \textbf{2 } $\b_{i+1} u_{i+1} = A v_i - \a_i u_i$
\ENDFOR
\end{algorithmic}
\end{algorithm}

The scalars $\a_i, \, \b_i$ are chosen to normalize the generated
vectors $v_i, u_i$. Forming the matrices $U_{k+1} = [u_1\, \cdots
\,u_{k+1}] \in \RR^{m \x (k+1)}$ and $V_{k} = [v_1\, \cdots\, v_k]\in
\RR^{n \x k}$, it is straightforward to show that
\begin{equation}
  \label{eq:GK-fact-1}
A V_k = U_{k+1} B_{k+1}, \qquad A\tp U_k = V_k \widehat{B}_k,
\end{equation}
where $V_k\tp V_k = I, \; U_{k+1}\tp U_{k+1} = I,$ and
\begin{equation}
\label{eq:Bk}
B_{k+1} =
\begin{bmatrix}
  \a_1 &   &   &   \\
  \b_2 & \a_2 &   &   \\
    & \ddots & \ddots &   \\
    &   & \b_k & \a_k \\
    &   &   & \b_{k+1} \\
\end{bmatrix}
=
\begin{bmatrix}
  \widehat{B}_k \\
\beta_{k+1} e_k\tp
\end{bmatrix}
\in \RR^{(k+1) \x k}
\end{equation}
is bidiagonal\footnote{Note that the two sequences of vectors become
orthogonal automatically; this is due to the fact that the
bidiagonalization procedure is equivalent to the Lanczos process
applied to the two symmetric matrices $A A\tp$ and $A\tp A$.}.

Using tensor notation from Section \ref{sec:ten-matmult}, and a special
case of the identification \eqref{eq:ident-13},   we may
express the two steps of the recursion as

\begin{algorithm}
\caption{Golub-Kahan bidiagonalization in tensor notation}
\label{alg:gk-bidiag-tensor}
\begin{algorithmic}
\FOR{$i = 1,2,\dots,k$}
\STATE \textbf{1 } $\a_i v_i = \tmr[1]{A}{u_i} - \b_i v_{i-1}$
\STATE \textbf{2 } $\b_{i+1} u_{i+1} = \tmr[2]{A}{v_i} - \a_i u_i$
\ENDFOR
\end{algorithmic}
\end{algorithm}

%In this notation the Krylov factorizations \eqref{eq:GK-fact-1} become
%%
%\begin{equation}
%  \label{eq:gk-fact}
%\tmr[2]{A}{V_k} = \tml[1]{U_{k+1}}{B_{k+1}}.
%\end{equation}
%
We observe that the $u_i$ vectors ``live'' in the first mode  of $A$,
and we generate the sequence $u_2, u_3, \ldots$, by 
multiplication of the $v_i$ vectors in the second mode, and vice
versa.

%\begin{rem}
%  Using the multilinear tensor-matrix product definition \ref{??} the
%  result of $\tmr[1]{A}{u_j}$ is in fact a \textit{row}-vector. For
%  simplicity we will use the convention that vectors $u_j, v_j$ as
%  well as $\tmr[2]{A}{v_j}$ and $\tmr[1]{A}{u_j}$ are all
%  \textit{column}-vectors. This \textit{inconvenience} becomes even
%  more apparent with third order tensors which involve rows, columns
%  and third mode \textit{fibers}. A similar convention will be made
%  for tensors as well.
%\end{rem}

%The columns of the matrices $U_k$ and $V_k$ from the Golub-Kahan
%bidiagonalization form orthonormal bases for the Krylov subspaces
%\begin{align*}
%\cK_k( A  A\tp, u_0) & = \spn \{u_0, (A A\tp) u_0, (A A\tp)^2 u_0,
%\dots, (A A\tp)^{k-1}u_0 \}, \\
%\cK_k(A\tp A ,v_0) & = \spn \{v_0, (A\tp A) v_0, ( A\tp A)^2 v_0,
%\dots, (A\tp A)^{k-1}v_0 \}.
%\end{align*}

\section{Tensor Krylov Methods}
\label{sec:tensor-krylov}

\subsection{A Minimal Krylov Recursion}
\label{sec:minimal}

In this subsection we will present the main process for the tensor Krylov
methods. We will further prove that, for tensors with $\rank(\cA) =
(p,q,r)$, we can capture all three subspaces associated to $\cA$
within $\max\{p,q,r\}$ steps of the algorithms. Finally we will state
a partial factorization that is induced by the procedure.  

Let $\cA \in \RR^{l \times m \times n}$ be a given tensor of order
three. It is now straightforward to generalize the Golub-Kahan
procedure, starting from Algorithm 
\ref{alg:gk-bidiag-tensor}. Assuming we have two  
%\footnote{It is not necessary to have three starting vectors in order to
%generate the sequences---two starting vectors is sufficient. Using three
%vectors simplifies the presentation.}  
starting vectors, $u_1 \in \RR^l$ and $v_1 \in \RR^{m}$ we can obtain a third mode vector $w_1 = \tmr[1,2]{\cA}{u_{1},v_{1}} \in \RR^n$. We can then generate three sequences of vectors
\begin{align}
u_{i+1} & = \tmr[2,3]{\cA}{v_{i},w_{i}}, \label{eq:ui}\\
v_{i+1} & = \tmr[1,3]{\cA}{u_{i},w_{i}}, \label{eq:vi}\\
w_{i+1}& = \tmr[1,2]{\cA}{u_{i},v_{i}}, \label{eq:wi}
\end{align}
for $ i = 1,\dots,k$. 
We see that the first mode sequence of vectors $(u_{i+1})$ are generated by
multiplication of second and third mode vectors $(v_i)$ and $(w_i)$ by the
tensor $\cA$, and  similarly for the other two sequences. The newly generated
vector is immediately orthogonalized against all the previous ones in its
mode, using the modified Gram-Schmidt process. An obvious alternative  to 
 \eqref{eq:vi} and \eqref{eq:wi} that
is consistent with the Golub-Kahan recursion is  to use the most
recent vectors in computing the new one. 
%The computation become   
%$v_{i+1}  = \tmr[1,3]{\cA}{u_{i+1},w_{i}}$ and $w_{i+1} =
%\tmr[1,2]{\cA}{u_{i+1},v_{i+1}}$.  
This recursion is presented in Algorithm~\ref{alg:minKrylov}. 
In the algorithm description it is
understood that $U_{i}=[u_1\, u_2\, \cdots \, u_{i}]$, etc.  
The coefficients ~$\a_u$, $\a_v$, and $\a_w$   are used to normalize the
generated vectors to length one. 

%Note also that we use the convention
% from Section \ref{sec:concepts} not to distinguish between column,
% row, and mode-3  vectors, but we write all as column vectors. 

For
reasons that will become clear later, we will refer to this recursion
as a minimal Krylov recursion.
%\begin{algorithm}
%\caption{Minimal Krylov recursion}
%\label{alg:minKrylov}
%\begin{algorithmic}
%\STATE Given: two normalized starting vectors $u_1$ and $v_1$, and
%$\RR^{k \times k \times k} \ni \cH=0$ 
%\STATE $\a_w w_1 = \tmr[1,2]{\cA}{u_1,v_1}$
%\FOR {$i=1,2,\ldots,k-1$}
%\STATE  $\widehat{u}=\tmr[2,3]{\cA}{v_i,w_i};\quad 
%         h_u= U_{i}\tp \widehat{u}$
%\STATE  $ \alpha_u u_{i+1} =
%               \widehat{u} - U_{i} h_u$; \quad 
%  $\cH(1:i+1,i,i)=\begin{bmatrix} h_u \\ \alpha_u \end{bmatrix}$     

%\STATE  $\widehat{v}=\tmr[1,3]{\cA}{u_{i+1},w_i}; \quad
%             h_v=V_{i}\tp \widehat{v}$
%\STATE  $\a_v v_{i+1} =
%               \widehat{v} - V_{i} h_v$; \quad 
%  $\cH(i+1,1:i+1,i)=\begin{bmatrix} h_v \\ \alpha_v \end{bmatrix}$     
%\STATE  $\widehat{w}=\tmr[1,2]{\cA}{u_{i+1},v_{i+1}}; \quad
%               h_w=W_{i}\tp \widehat{w}$
%\STATE  $\a_w w_{i+1} = 
%               \widehat{w} - W_{i} h_w$; \quad 
%  $\cH(i+1,i+1,1:i+1)=\begin{bmatrix} h_w \\ \alpha_w \end{bmatrix}$     
%\ENDFOR
%\end{algorithmic}
%\end{algorithm}
\begin{algorithm}
\caption{Minimal Krylov recursion}
\label{alg:minKrylov}
\begin{algorithmic}
\STATE Given: two normalized starting vectors $u_1$ and $v_1$, 
%and $\RR^{k \times k \times k} \ni \cH=0$ 
\STATE $\a_w w_1 = \tmr[1,2]{\cA}{u_1,v_1}$
\FOR {$i=1,2,\ldots,k-1$}
\STATE  $\widehat{u}=\tmr[2,3]{\cA}{v_i,w_i};\quad 
         h_u= U_{i}\tp \widehat{u}$
\STATE  $ \alpha_u u_{i+1} =
               \widehat{u} - U_{i} h_u$; \quad 
  $H_{i}^{u}=\begin{bmatrix} H_{i}^{u} & h_u \\ 0 & \alpha_u \end{bmatrix}$     

\STATE  $\widehat{v}=\tmr[1,3]{\cA}{u_{i+1},w_i}; \quad
             h_v=V_{i}\tp \widehat{v}$
\STATE  $\a_v v_{i+1} =
               \widehat{v} - V_{i} h_v$; \quad 
  $H_{i}^{v}=\begin{bmatrix} H_{i}^{u} & h_v \\ 0 & \alpha_v \end{bmatrix}$     
\STATE  $\widehat{w}=\tmr[1,2]{\cA}{u_{i+1},v_{i+1}}; \quad
               h_w=W_{i}\tp \widehat{w}$
\STATE  $\a_w w_{i+1} = 
               \widehat{w} - W_{i} h_w$; \quad 
  $H_{i}^{w}=\begin{bmatrix} H_{i}^{w} &  h_w \\ 0 & \alpha_w \end{bmatrix}$     
\ENDFOR
\end{algorithmic}
\end{algorithm}

  The process may  break down, i.e. we obtain a
new vector $u_{i+1}$, for instance,  which is linear combination of the
vectors in $U_{i}$. This can happen in two principally different
situations. The first one is when, for example, the vectors in $U_{i}$
span the range space of $A^{(1)}$. If this is the case we are done
generating new $u$-vectors. The second case is when we get a ``true
breakdown''\footnote{In the matrix case a breakdown occurs in the Krylov
recursion for instance if the matrix and the starting vector have the structure 
\[
A=
\begin{bmatrix}
  A_1 & 0 \\ 0 & A_2
\end{bmatrix}, \qquad 
v=
\begin{bmatrix}
  v_1 \\ 0
\end{bmatrix}.
\]
 An analogous situation
can occur with tensors.}, $u_{i+1}$ is a linear combination of vectors in $U_{i}$,
but $U_{i}$ does not span the entire range space of $A^{(1)}$. This
can be  fixed by taking a vector $u_{i+1} \perp U_{i}$ with $u_{i+1}$
in range of $A^{(1)}$.

\subsubsection{Tensors with Given Cubical Ranks}
\label{sec:cubical-rank}
Assume that the $l \x m \x n$ tensor has a cubical low rank,
i.e. $\rank(\cA) = (r,r,r)$ with $r \leq \min\{l, m, n \}$.  
Then there exist a tensor $\cC \in \RR^{r \x r \x r}$, and full
column rank matrices  $X,Y,Z$ such that    $\cA =
\tml{X,Y,Z}{\cC}$. 

We will now prove that, when the starting vectors $u_{1}$, $v_{1}$ and $w_{1}$ are in the range of the respective subspaces, the minimal Krylov procedure generates matrices $U,V,W$,
such that $\spn(U) = \spn(X)$, $\spn(V) = \spn(Y)$ and
$\spn(W) = \spn(Z)$ after $r$ steps. Of course,  for the low
multilinear rank approximation problem of tensors it is the subspaces
that are important, not their actual representation. The specific
basis spanning e.g. $\spn(X)$ is ambiguous.  

%We will now 
%prove that the minimal Krylov procedure generates matrices $U,V,W$
%such that $\spn(X) \subseteq \spn(U)$, $\spn(Y) \subseteq \spn(V)$ and
%$\spn(Z) \subseteq \spn(W)$ after $r$ or $r+1$  steps depending on the
%starting vectors.  Of course,  for the low
%multilinear rank approximation problem of tensors it is the subspaces
%that are important, not their actual representation. The specific
%basis spanning e.g. $\spn(X)$ is ambiguous.
%\todo{Det förenklar mycket om vi tar startvektorerna i range}
%\todo{Jag haller med, dock ar det inte alltid som man gor so. I experimenten kor vi t.ex. en hel del med slump-start-vectorer. Vi kan ha kvar detta sa lange.}
 
\begin{theorem}
\label{thm:minKrylov1}
Let $\cA = \tml{X,Y,Z}{\cC} \in \RR^{l \x m \x n}$ with  $\rank(\cA) =
(r,r,r)$. Assume we have starting vectors in the associated range spaces, i.e. 
$u_{1} \in \spn(X)$, $v_{1} \in \spn(Y)$, $w_{1} \in \spn(Z)$. Assume also that the process does not break down\footnote{The newly generated vector is not a linear
  combination of previously generated vectors.} within $r$ iterations. Then the minimal Krylov procedure in Algorithm \ref{alg:minKrylov}  generates matrices
$U_{r},V_{r},W_{r}$ with  
\[
\spn(U_{r}) = \spn(X), \quad \spn(V_{r}) = \spn(Y), \quad \spn(W_{r})= \spn(Z).
\]
%The vectors generated of Algorithm~\ref{alg:minKrylov} capture the sought subspaces of $\cA$ in $r$ iterations.
\end{theorem}

%\begin{theorem}
%\label{thm:minKrylov1}
%Let $\cA = \tml{X,Y,Z}{\cC} \in \RR^{l \x m \x n}$ with  $\rank(\cA) =
%(r,r,r)$. Assume we have random starting vectors $u_{1} \in \RR^{l}$,
%$v_{1}\in \RR^{m}$ and $w_{1} \in \RR^{n}$. Assuming the process does
%not break down\footnote{The newly generated vector is not a linear
%  combination of previously generated vectors.} within $r+1$ steps, the minimal Krylov
%procedure in Algorithm \ref{alg:minKrylov} will generate matrices
%$U_{r+1},V_{r+1},W_{r+1}$ with  
%\[
%\spn(X) \subseteq \spn(U_{r+1}), \quad \spn(Y) \subseteq \spn(V_{r+1}), \;  \text{ and } \; \spn(Z) \subseteq \spn(W_{r+1}),
%\]
%i.e. after $r+1$ steps we capture the sought subspaces of $\cA$.  If it holds that $u_{1} \in \spn(X)$, $v_{1} \in \spn(Y)$, $w_{1} \in \spn(Z)$ and the process does not break down within $r$ steps, then the minimal Krylov procedure will generate matrices $U_{r},V_{r},W_{r}$ that span the same subspace as $X,Y,Z$, respectively. The generated vectors of Algorithm~\ref{alg:minKrylov} capture the sought subspaces in $r$ steps.
%\end{theorem}

\begin{proof}
First observe that the recursion generates vectors in the span of $X$,
$Y$, and $Z$, respectively:  
\begin{align*}
\tmr[2,3]{\cA}{v,w} & = \tmr{\cC}{X\tp,Y\tp v,Z\tp w} =  \tmr{\cC}{X\tp,\bar{v},\bar{w}} = X c_{1},\\
\tmr[1,3]{\cA}{u,w} & = \tmr{\cC}{X\tp u,Y\tp,Z\tp w} =  \tmr{\cC}{\bar{u},Y\tp,\bar{w}} = Y c_{2}, \\
\tmr[1,2]{\cA}{u,v} & = \tmr{\cC}{X\tp u,Y\tp v,Z\tp} =
\tmr{\cC}{\bar{u},\bar{v},Z\tp} = Z c_{3}, 
\end{align*}
where in the first equation $\bar{v} = Y\tp v$, $\bar{w} = Z\tp w$ and $c_1 =
\tmr[2,3]{\cC}{\bar{v},\bar{w}}$, and  the other two equations are
analogous. Consider the first  mode vector $u$. Clearly it  is a
linear combination of the column vectors in $X$. Since we 
orthonormalize every newly generated $u$-vector against all the
previous vectors, and since we  assume that the process does not break
down, it follows that $\dim(\spn([u_{1} \, \cdots \,u_{k}]) ) = k$ for
$k \leq r$ will increase by one with every new $u$-vector. 
Given that $u_{1} \in \spn(X)$ then for $k = r$ we have that $\spn([u_{1}\,\cdots
\,u_{r}])= \spn(X)$. The  proof is analogous for the second and third modes.
\end{proof}

We would like to make e few remarks on this theorem:
\paragraph{Remark (1)} 
It is straightforward to show that when the starting vectors are not in the
associated range spaces we would only need to do one more iteration, i.e. in total $r+1$ iterations, to obtain matrices $U_{r+1}$, $V_{r+1}$ and $W_{r+1}$ that would span the column spaces of $X$, $Y$ and $Z$, respectively. 
\paragraph{Remark (2)} It is easy to obtain
starting vectors $u_{1} \in \spn(X)$, $v_{1} \in \spn(Y)$ and $w_{1}
\in \spn(Z)$. Choose any single nonzero mode-$k$ vector or the mean of
the mode-$k$ vectors. 
\paragraph{Remark (3)} Even if we do not choose starting vectors in
the range spaces of $X,Y,Z$ and run the minimal Krylov procedure $r+1$
steps we can easily obtain a matrix $U_{r}$ spanning the correct
subspaces. To do this just observe that $U_{r+1}\tp A^{(1)} = U_{r+1}\tp X
C^{(1)}(Y \otimes Z)\tp$ is an $(r+1) \x mn $ matrix with rank~$r$.

\subsubsection{Tensors with General Low Multilinear Rank}
\label{sec:general-rank}
Next we  discuss the case when the tensor $\cA \in \RR^{l \x m \x
  n}$ has $\rank(\cA) = (p,q,r)$ with $p<l$, $q<m,$ and $r < n$.
Without loss of generality we can assume $ p \leq q \leq r$.  
 Then $\cA = \tml{X,Y,Z}{\cC}$ where $\cC$ is
a $p \x q \x r$ tensor and $X,Y,Z$ are full column rank matrices with
conformal dimensions. The 
discussion assumes exact arithmetic and that no breakdown occurs.%
%\todo[color=green]{Jag flyttade ekvationerna till förra satsen, mer naturligt.}
%In summary, since the result of tensor $\x$ vector $\x$ vector
%products are linear combinations of vectors in $X,Y,Z$, respectively,
%range spaces for these matrices can be obtained in a total of $p$
%steps for $X$, $q$ steps for $Y$ and $r$ steps for $Z$.  
%
% We can use Theorem \ref{thm:minKrylov1} to obtain a
%  rank-$(p,p,p)$ approximation of the tensor. 
%  First, we want to point
%  out, a newly generated vector that is a linear combination of
%  previous vectors can be detected, without any additional
%  computational cost, by monitoring the normalization
%  coefficient $\a_{u},\a_{v},\a_{w}$ in
%  Algorithm~\ref{alg:minKrylov}. Detecting that the dimension of
%  $\spn(U_{k})$, the 
%  first mode vectors, does not increase (at step $ k = p+1$), we will
%  stop generating new $u$-vectors. 
%  Since we no longer need to generate first mode vectors, for the remaining
%  $q-p$ second mode vectors and $r-p$ third mode vectors, we need to
%  make some modifications to the procedure outlined in
%  Algorithm~\ref{alg:minKrylov}. 

From the proof of Theorem \ref{thm:minKrylov1} we see that the vectors
generated are in the span of $X$, $Y$, and $Z$, respectively.    
 Therefore, after  having performed $p$ steps we
will not be able to generate any new vector in the first mode. This
can be detected from the fact that the result of the orthogonalization
is zero. We can now continue generating vectors in the second and
third modes, using any of the first mode vectors, or a (possibly
random) linear combination of them\footnote{Also the optimization
  approach of Section \ref{sec:kryl-opt} can be used.}. This can be
repeated until we have generated $q$ vectors in the second and third
modes. The final $r-q$ mode-3 vectors can then be generated using
combinations of mode-1 and mode-2 vectors that have not been used
before, or, alternatively, random linear combinations of previously
generated mode-1 and mode-2 vectors.  We refer to the procedure
described in this paragraph as the \emph{modified minimal Krylov
  recursion}.

\begin{theorem}
\label{thm:minKrylov22}
Let $\cA \in \RR^{l \x m \x n}$ be a tensor of $\rank(\cA) = (p,q,r)$
with $p \leq q \leq r$. We can then write $\cA = \tml{X,Y,Z}{\cC}$,
where $\cC$ is a $p \x q \x r$ tensor and $X,Y,Z$ are full column rank
matrices with conforming dimensions. Assume that the starting vectors
satisfy $u_{1} \in \spn{(X)}$, $v_{1} \in \spn{(Y)}$ and $w_{1} \in
\spn{(Z)}$.  Assume also that the process does not break down except
when we obtain a set of vectors spanning the full range spaces of the
different modes.  Then in exact arithmetic, and in a total of $r$
steps the modified  minimal Krylov recursion produces matrices
$U_{p}$, $V_{q}$ and $W_{r}$, which span the same
subspaces as $X,Y$, and $Z$, respectively.
\end{theorem}

The actual numerical implementation of the procedure in floating point
arithmetic is, of course, much more complicated. For instance, the
ranks will never be exact, so one must devise a criterion for
determining the numerical ranks that will depend on the choice of
tolerances. This will be the topic of our future research. 

\subsubsection{Partial Factorization}
\label{sec:partial-fact}
To our knowledge there is no simple way of writing the minimal Krylov
recursion directly as a tensor Krylov factorization, analogous to
\eqref{eq:GK-fact-1}.  However, having generated 
three orthonormal matrices $U_k$, $V_k,$ and $W_k$, we can easily compute a
low-rank tensor approximation of $\cA$ using Lemma~\ref{lem:LS}, 
\begin{equation}
  \label{eq:low-rank}
\cA \approx \tml{U_k,V_k,W_k}{\cH}, \qquad 
\cH= \tml{U_k\tp,V_k\tp,W_k\tp}{\cA} \in \RR^{k \times k \times
  k}.  
\end{equation}
Obviously, $\cH_{\lambda\mu\nu}=\tmr{\cA}{u_\lambda,v_\mu,w_\nu}$. 
Comparing with Algorithm \ref{alg:minKrylov} we see that $\cH$
contains elements from the Hessenberg matrices
$H^{u},H^{v},H^{w}$, which  contain the
orthogonalization and normalization coefficients. However, not all the
elements in $\cH$ are generated in the recursion, only those that are
close to the ``diagonals''. Observe also that $\cH$ has $k^{3}$
elements, whereas the minimal Krylov procedure generates three
matrices with total number of $3k^{2}$ elements. We now show   that
the minimal Krylov procedure induces a certain 
\textit{partial tensor-Krylov factorization}.

\begin{proposition}\label{prop:minK-fact}
  Assume that $U_k$, $V_k$, and $W_k$   have been generated by the
  minimal Krylov recursion and that $\cH =
  \tmr{\cA}{U_k,V_k,W_k}$. Then, for $1 \leq i \leq k-1$, 
  \begin{align}
    (\tmr[2,3]{\cA}{V_k,W_k})(:,i,i) &= (\tml[1]{U_k}{\cH})(:,i,i) = U_{k} H^{u}(:,i),
    \label{eq:minK-factU} \\ 
    (\tmr[1,3]{\cA}{U_k,W_k})(i+1,:,i) &= (\tml[2]{V_{k}}{\cH})(i+1,:,i) = V_{k}H^{v}(:,i),
    \label{eq:minK-factV}\\ 
    (\tmr[1,2]{\cA}{U_k,V_{k}})(i+1,i+1,:) &= (\tml[3]{W_{k}}{\cH})(i+1,i+1,:) = W_{k} H^{w}(:,i).\label{eq:minK-factW}
  \end{align}
\end{proposition}

\begin{proof}
  Let $1 \leq i \leq k-1$ and consider the fiber 
\[
\cH(:,i,i)= [h_{2ii}\;\; h_{2ii} \;\; \cdots \;\; h_{i+1, i i} \;\;
  h_{i+2, i i} \;\; \cdots \;\; h_{k i i}]\tp
\]
Since, from the minimal recursion, 
\[
\tmr[2,3]{\cA}{v_i,w_i} = \sum_{\lambda=1}^{i+1} h_{\lambda ii}
u_\lambda = U_{i+1} H^{u}_{i}(:,i), 
\]
we have, for $i+2 \leq s \leq k$, 
\[
h_{sii} = \tmr{\cA}{u_s,v_i,w_i} =
\tml[1]{u_s\tp}{(\tmr[2,3]{\cA}{v_i,w_i})} =  0. 
\]
Thus  $  h_{i+2, i i}  =\ldots =    h_{k i i}=0.$ Therefore, the
fiber in the left hand side of \eqref{eq:minK-factU} is equivalent to the minimal
recursion for computing $u_{i+1}$. The rest of the proof is analogous.
\end{proof}

If the sequence of vectors is generated according to Equations \eqref{eq:ui}--\eqref{eq:wi}, then a similar (and simpler) proposition will hold. For example we would have 
\[
(\tmr[1,3]{\cA}{U_{k},W_{k}})(i,:,i) = \tml[2]{V_{k}}{\cH}(i,:,i) = V_{k} H^{v}(:,i), \quad i = 1, \dots,k.
\]

\subsection{A Maximal Krylov Recursion}
\label{sec:maximal}

Note that when a new $u_{i+1}$ is generated in the minimal Krylov
procedure, then we use the most recently computed $v_i$ and
$w_i$.  In fact, we might choose any combination of previously
computed $\{v_1,v_{2}, \dots, v_{i}\}$ and $\{w_1,\dots,w_{i}\}$ that have not been used before to generate a $u$-vector.
Let $\ j \leq i$ and $k \leq i$, and consider the
computation of a new $u$-vector, which we may write
\begin{algorithm}
\begin{algorithmic}
  \STATE $h_u=U_i\tp(\tmr[2,3]{\cA}{v_{j},w_{k}})$
  \STATE $h_{* j k} u_{i+1} =
               \tmr[2,3]{\cA}{v_{j},w_{k}} - U_{i} h_u$
\end{algorithmic}
\end{algorithm}

Thus if we are prepared to use all previously computed $v$- and
$w$-vectors, then   we have a much richer combinatorial structure, which we
illustrate in the following diagram. Assume that $u_1$ and $v_1$ are
given.  In the first steps of the maximal Krylov procedure the following
vectors can be generated by combining previous vectors.
\begin{align*}
& \textbf{1:} & \{u_{1}\}\x \{v_{1}\} \quad &\longrightarrow \quad w_{1} \\
& \textbf{2:} &\{v_{1}\} \x \{w_{1}\} \quad &\longrightarrow \quad u_{2} \\
& \textbf{3:} &\{u_{1},u_{2}\} \x \{w_{1}\} \quad &\longrightarrow \quad \{v_{2}\;v_{3}\} \\
& \textbf{4:} & \{u_{1},u_{2}\}\x \{v_{1},v_{2},v_{3}\} \quad &\longrightarrow \quad 
\{(w_{1}),w_{2},w_{3},w_{4},w_{5},w_{6}\}\\
& \textbf{5:} & \{v_{1},v_{2},v_{3}\}\x \{w_{1},w_{2},\dots,w_{6}\} \quad &\longrightarrow \quad 
\{(u_{2}),u_{3},\dots,u_{19}\} \\
& \textbf{6:} & \{u_{1},u_{2},\dots,u_{19}\}\x \{w_{1},w_{2},\dots,w_{6}\} \quad &\longrightarrow \quad 
\{(v_{2}),(v_{3}),v_{4},\dots,v_{115}\}
\end{align*}

Vectors computed at a previous step are within parentheses. Of course, we can
only generate new orthogonal vectors as long as the total number of vectors
is smaller than the dimension of that mode. Further, if at a certain
stage in the procedure we have generated $\a$ and $\b$ vectors in
two modes, then we can generate altogether $\g=\a\b$ vectors in
the third mode (where we do not count the starting vector in that
mode, if there was one).

We will now describe the first three steps in some detail. Two starting
vectors $u_{1}$ and $v_{1}$, in the first and second mode,
respectively. We also assume that $\| u_{1} \| = \| v_{1} \|=1$. The
normalization and orthogonalization coefficients will be stored in a
tensor $\cH$. Its entries are denoted with $h_{ijk} =
\cH(i,j,k)$. Also when subscripts are written on tensor $\cH$, they
will indicate the dimensions of the tensor, e.g. $\cH_{211}$ is a $2
\times 1 \times 1$ tensor. 

\paragraph*{Step (1)} In the first step we generate an new third mode vector by computing 
\begin{equation}
\label{eq:step1}
\tmr[1,2]{\cA}{u_{1},v_{1}} =  h_{111}w_{1} = \tml[3]{w_{1}}{\cH_{111}}, 
\end{equation}
where $h_{111} = \cH_{111}$ is a normalization constant. 
%We also 
%write this equality as a multilinear tensor matrix product to
%illustrate the tensor-Krylov factorization it represents.  
\paragraph{Step (2)} Here  we compute a new first mode vector; 
\begin{align*}
\widehat{u}_2  = \tmr[2,3]{\cA}{v_{1},w_{1}}. % =  \tml[3]{w_{1}}{\cH_{111}}
\end{align*}
The orthogonalization coefficient satisfies
\begin{equation}
  \label{eq:u20}
u_{1}\tp \widehat{u}_2 = u_1\tp (\tmr[2,3]{\cA}{v_{1},w_{1}}) = 
\tmr{\cA}{u_{1},v_{1},w_{1}} = w_1\tp(\tmr[1,2]{\cA}{u_{1},v_{1}}) =
h_{111},   
\end{equation}
 from  \eqref{eq:step1}. After orthogonalization and normalization, 
 \begin{equation}
   \label{eq:u2}
h_{211}u_{2}  = \widehat{u}_2  - h_{111}u_{1},   
 \end{equation}
and rearranging the terms in \eqref{eq:u2}, we  have the following
tensor-Krylov factorization 
\[
\tmr[2,3]{\cA}{v_{1},w_{1}} = \tml[1]{[u_{1},\,u_{2}]}{\cH_{211}},
\qquad \cH_{211} =
\begin{bmatrix}
  h_{111} \\
  h_{211} 
\end{bmatrix}.
\]

\paragraph{Step (3)} In the third step we obtain two second mode
vectors. To get $v_{2}$ we  compute  
\[
\widehat{v}_2 = \tmr[1,3]{\cA}{u_{1},w_{1}}, \qquad 
h_{121}v_{2} = \widehat{v}_2 - h_{111}v_{1};
\]
the orthogonalization coefficient becomes $h_{111}$ using an 
argument analogous to that  in \eqref{eq:u20}. 

 Combining $u_{2}$ with $w_{1}$ will
yield $v_{3}$ as follows;  first we compute
\[
\widehat{v}_3 = \tmr[1,3]{\cA}{u_{2},w_{1}}, 
\]
and orthogonalize 
\[
v_1\tp \widehat{v}_3 = \tmr{\cA}{u_2,v_1,w_1} = 
u_2\tp(\tmr[2,3]{\cA}{v_1,w_1}) = 
u_2\tp \widehat{u}_2 = h_{211}.
\]
We see from \eqref{eq:u2} that $h_{211}$ is already computed.
The second orthogonalization becomes 
\[
v_2\tp \widehat{v}_3 = \tmr{\cA}{u_2,v_2,w_1} =: h_{221}.
\]
Then 
\[
\qquad 
h_{231}v_{3} = \widehat{v}_3 -  h_{211}v_{1} - h_{221}v_{2} 
\]
After a completed third step we have a new tensor-Krylov factorization
\[
\RR^{2 \times m \times 1} \ni \tmr[1,3]{\cA}{[u_{1}\, u_{2}],w_{1}} =
\tml[2]{[v_{1} \,v_{2} 
  \,v_{3}]}{\cH_{231}}, \qquad  
\cH_{231} =
\begin{bmatrix}
  h_{111} & h_{121} & 0 \\
  h_{211} &   h_{221} &   h_{231} 
\end{bmatrix}.
\]
Note that the orthogonalization coefficients are given by 
\[
h_{\lambda\mu\nu} = \tmr{\cA}{u_\lambda,v_\mu,w_{\nu}}.
\]

\begin{algorithm}[ht!]
\caption{Maximal Krylov recursion}
\label{alg:max-krylov}
\begin{algorithmic}
\STATE $u_{1},v_{1}$ given starting vectors of length one
\STATE $h_{111} w_1 = \tmr[1,2]{\cA}{u_1,v_1}$
\STATE $\a=\b=\g=1$, $U_{\a} = u_{1}$, $V_{\b} = v_{1}$ and $W_{\g} = w_{1}$
\WHILE {$\a \leq \a_{\mathrm{max}}$ and $\b \leq
  \b_{\mathrm{max}}$ and $\g  \leq \g_{\mathrm{max}}$  }
\STATE \%-------------------- \emph{$u$-loop } --------------------\%
\STATE $U_{\a} = [u_{1},\dots,u_{\a}]$, $U = [\;]$,  $V_{\b} = [v_{1},\dots,v_{\b}]$,  $W_{\g} = [w_{1},\dots,w_{\g}]$, $i = 1$
\FORALL{$(\bar{\b},\bar{\g})$ such that $\bar{\b} \leq \b$ and
  $\bar{\g}  \leq \g$ }
\IF{the pair $(\bar{\b},\bar{\g})$ has not been used before}
\STATE $h_{\a} = \cH(1:\a,\bar{\b},\bar{\g})$
\STATE $h_{i} =\tmr{\cA}{U,v_{\bar{\b}},w_{\bar{\g}}}$
\STATE $h_{\a+i,\bar{\b}\bar{\g}} u_{\a+i} =
               \tmr[2,3]{\cA}{v_{\bar{\mu}},w_{\bar{\lambda}}} -
               U_{\a} h_a - U h_{i}$
\STATE $\cH(\a+1:\a+i,\bar{\b},\bar{\g}) = [h_{i}\tp \; \;h_{\a+i,\bar{\b} \bar{\g}}]\tp$
\STATE  $U = [U\, u_{\a + i}]$, $i=i+1$
\ENDIF
\ENDFOR
\STATE $U_{\b \g +1} = [U_{\a}\, U]$, $\a = \b \g +1$
\STATE \%-------------------- \emph{$v$-loop } --------------------\%
\STATE $U_{\a} = [u_{1},\dots,u_{\a}]$, $V_{\b} = [v_{1},\dots,v_{\b}]$, $V = [\;]$,  $W_{\g} = [w_{1},\dots,w_{\g}]$, $j = 1$
\FORALL{$(\bar{\a},\bar{\g})$ such that $\bar{\a} \leq \a$ and
  $\bar{\g}  \leq \g$ }
\IF{the pair $(\bar{\a},\bar{\g})$ has not been used before}
\STATE $h_{\b} = \cH(\bar{\a},1:\b,\bar{\g})$
\STATE $h_{j} =\tmr{\cA}{u_{\bar{\a}},V,w_{\bar{\g}}}$
\STATE $h_{\bar{\a},\b+j,\bar{\g}} v_{\b+j} =
               \tmr[1,3]{\cA}{u_{\bar{\a}},w_{\bar{\g}}} -
               V_{\b} h_\b - V h_{j}$
\STATE $\cH(\bar{\a},\b+1:\b+j,\bar{\g}) = [h_{j}\tp \; \;h_{\bar{\a},\b+j ,\bar{\g}}]\tp$
\STATE  $V = [V\, v_{\b + j}]$, $j=j+1$
\ENDIF
\ENDFOR
\STATE $V_{\a \g +1} = [V_{\b}\, V]$, $\b = \a \g +1$
\STATE \%-------------------- \emph{$w$-loop } --------------------\%
\STATE $U_{\a} = [u_{1},\dots,u_{\a}]$,  $V_{\b} = [v_{1},\dots,v_{\b}]$,  $W_{\g} = [w_{1},\dots,w_{\g}]$, $W = [\;]$, $k = 1$
\FORALL{$(\bar{\a},\bar{\b})$ such that $\bar{\a} \leq \a$ and
  $\bar{\b}  \leq \b$ }
\IF{the pair $(\bar{\a},\bar{\b})$ has not been used before}
\STATE $h_{\g} = \cH(\bar{\a},\bar{\b},1:\g)$
\STATE $h_{k} =\tmr{\cA}{u_{\bar{\a}},v_{\bar{\b}},W}$
\STATE $h_{\bar{\a}\bar{\b}, \g+k} w_{\g+k} =
               \tmr[1,2]{\cA}{u_{\bar{\a}},v_{\bar{\b}}} -
               W_{\g} h_\g - W h_{k}$
\STATE $\cH(\bar{\a},\bar{\b},\g+1:\g+k) = [h_{k}\tp \; \;h_{\bar{\a}\bar{\b}, \g+k}]\tp$
\STATE  $W = [W\, w_{\g + k}]$, $k=k+1$
\ENDIF
\ENDFOR
\STATE $W_{\a \b} = [W_{\g}\, W]$, $\g = \a \b $

\ENDWHILE
\end{algorithmic}
\end{algorithm} 
%We have chosen to simplify the details in the algorithm to ease the presentation. For example vectors that have already been computed are recomputed in the next \textbf{while}-iteration in order to simplify the indexing in the \textbf{for}-loops. This is why we set $U = u_{1}$, $V = v_{1}$ and $W = w_{1}$ in the beginning of the $u$-, $v$- and $w$-loops, correspondingly. 
This maximal procedure is presented in Algorithm \ref{alg:max-krylov}.
The algorithm has three main loops, and it is maximal in the sense
that in each such loop we generate as many new vectors as can be
done, before proceeding to the next main loop. 
Consider the $u$-loop (the other loops are analogous). 
The vector $h_\a$ is a mode-1 vector\footnote{We here
refer to the identification \eqref{eq:ident-13}.} and contains orthogonalization coefficients with respect to $u$-vectors computed at previous steps. These coefficients are values of the tensor $\cH$. The vector $h_{i}$ on the other hand contains orthogonalization coefficients with respect to $u$-vectors that are computed within the current step.  
Its dimension is equal to the current number of vectors in $U$. The
coefficients $h_{i}$ together with the normalization constant
$h_{\a+1,\bar{\b},\bar{\g}}$ of the newly generated vector $u_{\a+i}$
are appended at the appropriate positions of the tensor
$\cH$. Specifically the coefficients for the $u$-vector obtained using
$v_{\bar{\b}}$ and $w_{\bar{\g}}$ are stored as first mode fiber,
i.e. $\cH(:,\bar{\b},\bar{\g}) = [h_{\a}\tp \;\; h_{i}\tp \;\;
h_{\a+i,\bar{\b},\bar{\g}}]\tp$. 
Since the number of vectors in $U$ are increasing for every new $u$-vector the dimension of 
$[h_{\a}\tp \;\; h_{i}\tp \;\; h_{\a+i,\bar{\b},\bar{\g}}]\tp$ and thus the dimension of $\cH$ along the first mode increases by one as well. The other mode-1 fibers are filled out with a zero at the bottom. Continuing with the $v$-loop, the dimension of the coefficient tensor $\cH$ increases in the second mode.

It is clear that $\cH$ has a zero-nonzero structure that
resembles that of a Hessenberg matrix. 
%In fact matricizing\footnote{Any tensor can be reshaped to a matrix with suitable dimensions. Specifically a third order tensor $\cA \in \RR^{l \x m \x n}$ can be reshaped to matrices $A^{(1)}\in \RR^{l \x mn}$, $A^{(2)}\in \RR^{m \x ln}$ and  $A^{(3)}\in \RR^{n \x lm}$, consider \cite{elsa:07,bako:06,a,eccv:02} for~details.} $\cH$ in any mode will yield a Hessenberg matrix, i.e. the matricizations $H^{(1)}$, $H^{(2)}$ and $H^{(3)}$ are all of Hessenberg form.
If we break the recursion after any complete outer \textbf{for all}-statement,
we can form a tensor-Krylov factorization.

\begin{theorem}[Tensor Krylov factorizations]\label{theo:tensor-fact}
Let a tensor $\cA \in \RR^{l \times m \times n}$ and two starting
vectors $u_{1}$ and $v_{1}$ be given. Assume that we have generated matrices with orthonormal columns using the maximal Krylov procedure of Algorithm 
\ref{alg:max-krylov}
, and a tensor $\cH$ of orthonormalization coefficients. Assume that after a complete $u$-loop
the matrices $U_\a$, $V_\b$, and  $W_\g$, and the tensor  $\cH_{\a \b \g}
\in \RR^{\a \times \b \times \g}$, have been generated,  where
$\a \leq l$, $\b \leq m$,  and $\g \leq n$. Then
  \begin{equation}\label{eq:tens-fact-U}
    \tmr[2,3]{\cA}{V_\b,W_\g} = \tml[1]{U_\a}{\cH_{\a \b \g}}.
  \end{equation}
Further, assume that after the following complete $v$-loop
we have orthonormal matrices $U_{\a}$, $V_{\bar \b}$, $W_{\g}$, and the 
tensor $\cH_{\a \bar \b \g} \in \RR^{\a \times \bar\b \times \g}$ where $\bar \b = \a \g +1  > \b$. Then
\begin{equation}\label{eq:tens-fact-V}
      \tmr[1,3]{\cA}{U_\a,W_\g} = \tml[2]{V_{\bar \b}}{\cH_{\a \bar \b \g}}.
\end{equation}
Similarly, after the following complete
$w$-loop, we will have orthonormal matrices $U_{\a}$, $V_{\bar \b}$, $W_{\bar \g}$
and the tensor $\cH_{\a \bar \b \bar \g} \in \RR^{\a \times \bar\b \times \bar \g}$ where $\bar \g = \a \bar \b   > \g$. Then%
\begin{equation}\label{eq:tens-fact-W}
      \tmr[1,2]{\cA}{U_\a,V_{\bar\b}} = \tml[3]{W_{\bar \g}}{\cH_{\a \bar \b \bar \g}}.
\end{equation}
It also holds that $\cH_{\a \b\g} = \cH_{\a \bar \b \g} (1:\a, 1:\b, 1:\g)$ and $\cH_{\a \bar \b \g} = \cH_{\a \bar \b \bar \g} (1:\a, 1:\bar \b, 1:\g)$, i.e. all orthonormalization coefficients from the $u$-, $v$- and $w$-loops are stored in a single and common tensor $\cH$.
\end{theorem}

\begin{proof}
  We prove that \eqref{eq:tens-fact-U} holds; the other two equations
  are analogous. Using the definition of matrix-tensor multiplication
  we see that $\tmr[2,3]{\cA}{V_\b,W_\g}$ is a tensor in $\RR^{l \times
    \b \times \g}$, where the first mode fiber at position
  $(j,k)$ with $j \leq \b$ and $ k \leq \g$ is given by $
  \widehat{u}_{\lambda} = 
  \tmr[2,3]{\cA}{v_{j},w_{k}}$ with $\lambda = (j-1)\g + k + 1$.

On the right hand side the corresponding first mode fiber $\cH(:,j,k)$
is equal to  
\[
\begin{bmatrix}
  h_{1jk}\\
  h_{2jk}\\
\vdots \\
  h_{\lambda-1 jk}\\
  h_{\lambda jk}\\
   \mathbf{0}
\end{bmatrix}
= 
\begin{bmatrix}
  \tmr{\cA}{u_1,v_{j},w_{k}} \\
  \tmr{\cA}{u_2,v_{j},w_{k}} \\
\vdots \\
  \tmr{\cA}{u_{\lambda-1},v_{j},w_{k}} \\
   h_{\lambda j k } \\
   \mathbf{0}
\end{bmatrix}.
\]
Thus we have 
\[
\widehat{u}_\lambda = \tmr[2,3]{\cA}{v_{j},w_{k}} = 
\sum_{i=1}^\lambda h_{ijk} u_i,
\]
which is the equation for computing $u_\lambda$ in the algorithm. 
\end{proof}

Let $U_j$ and $V_k$ be two matrices with orthonormal columns that have
been generated by any tensor Krylov method (i.e., not necessarily a
maximal one) with tensor $\cA$.  Assume that we then generate a
sequence of $m=jk$ vectors $(w_1, w_2,\ldots,w_{m})$ as in the $w$-loop of
the maximal method. From the proof of Theorem \ref{theo:tensor-fact}
we see that we have a tensor-Krylov factorization of the type
\eqref{eq:tens-fact-W},
\begin{equation}
\label{eq:k-fac}
      \tmr[1,2]{\cA}{U_j,V_k} = \tml[3]{W_m}{\cH_{jkm}}.
\end{equation}

It is clear that the dimensions of the $U$, $V$ and $W$ in the maximal
Krylov recursion become very large even after only 6--7 steps of the
procedure. It is not clear how preserving a tensor-Krylov
factorization can be utilized in practical implementation
applications. For the matrix case the theory of Krylov factorizations
is very important in enabling efficient implementation for various
algorithms.  However, the maximal Krylov recursion suggests a way for
an efficient algorithmic implementation. Consider the multilinear
approximation problem of an $l \x m \x n$ tensor $\cA$
\[
\min_{U,V,W,\cS} \| \cA - \tml{U,V,W}{\cS} \|
\]
which is equivalent to 
\[
\max_{U,V,W} \| \tmr{\cA}{U,V,W} \|, \quad U\tp U = I, \; V\tp V = I,\; W\tp W = I
\]
and $U,V,W$ are $l \x p$, $m \x q$ and $n\x r $ orthonormal matrices, respectively. 
Assume that we have generated $V_{\b}=[v_{1} \; \dots \;v_{\b}]$ and
$W_{\g}=[w_{1} \; \dots \;w_{\g}]$ using the maximal Krylov recurrence but
the number of combinations of $v$- and $w$-vectors exceeds the number
of $u$-vectors that are desired, i.e. $\b \g > p$. A natural thing to
do in this case is to compute the product
$\mathcal{U}_{\b\g}=\tmr[2,3]{\cA}{V_{\b},W_{\g}}$ and compute the $p$
dimensional dominant subspace of its mode one matricization
$U_{\b\g}^{(1)}$. Similarly for the other modes. With this
modification we no longer have a tensor-Krylov factorization, however
we can manage the blow up in the size of the dimensions for $U,V,W$
and obtain efficient algorithms. 
Although natural, this approach may still be impractical. For example if $l
=10^{4}$, and $\beta = \gamma = 100$, then $U_{\beta \gamma}^{(1)}$ will be
a large and dense $10^{4} \x 10^{4}$ matrix. If we are interested in an
approximation with $p = 100$ (rank of the first mode in the approximation)
an alternative to compute the dominant 100 dimensional subspace of
$U_{\beta \gamma}^{(1)}$ would be to take dominant (in some sense)
$\bar{V}_{10}$ and $\bar{W}_{10}$ subspaces of $V_{\beta}$ and
$W_{\gamma}$, respectively, and  compute $\bar{\mathcal{U}}_{100} =
\tmr[2,3]{\cA}{\bar{V}_{10},\bar{W}_{10}}$. Then $U_{100}$ is obtained
from the columns of $\bar{\mathcal{U}}_{100}$. 
%\todo[color=green]{Stycket ovan är för otydligt och spekulativt, kan strykas?} 

\subsection{Optimized  Minimal  Krylov Recursion}
\label{sec:kryl-opt}

In some applications it may be a disadvantage that the maximal Krylov
method generates so many vectors in each mode. In addition, when applied as
described in Section \ref{sec:maximal} it generates different numbers of
vectors in the different modes. Therefore it is natural to ask whether one
can modify the minimal Krylov recursion so that it uses ``optimal'' vectors
in two modes for the generation of a vector in the third mode. Such
procedures have recently been suggested in \cite{gos10}.
We will describe this approach in terms of the recursion  of a vector in
the  mode 3. The corresponding computations in  modes 1 and 2 are
analogous. 

Assume that we have computed $i$  vectors in the first two modes, for
instance, and that we are about to compute $w_i$. Further, assume that we
will use  linear combinations of the vectors from modes 1 and 2, i.e. we
compute
\[
 \widehat{w}=\tmr[1,2]{\cA}{U_i \theta,V_i \eta},
\]
where $\theta, \eta \in \RR^i$ are yet to be specified. We want the new
vector to enlarge the ``$W$'' subspace as much as possible. This is the
same as requiring that $w_i$ be as large (in norm) as possible under the
constraint that it is orthogonal to the previous mode-3 vectors. Thus we
want to solve
\begin{align}
\max_{\theta,\eta}  \| \widehat{w} \|, \; \text{ where } \;  &\widehat{w} =
\tmr[1,2]{\cA}{U_i \theta,V_i \eta}, \label{eq:Copt}\\ 
  &\widehat{w} \perp W_{i-1}, \quad  \|
\theta \| = \| \eta \| = 1, \quad \theta,\eta \in \RR^{i}. \nonumber
\end{align}
The solution of  this problem is obtained by computing the best
rank-$(1,1,1)$ approximation $\tml{\theta,\eta,\omega}{\cS}$ 
of the tensor 
\begin{equation}
  \label{eq:Cw}
\cC_{w} = \tmr{A}{U_{i},V_{i},I-W_{i-1}W_{i-1}\tp}. 
\end{equation}
A suboptimal solution can be obtained from  the  HOSVD of $\cC_{w}$.

%For completeness we will state the corresponding problems in the two
%other modes:  
%\[
%\max_{\theta,\eta}  \| u \|, \; \text{ where } \; u =
%\tmr[2,3]{\cA}{V_i \theta,W_{i+1} \eta}  
%\]
%and such that $ u \perp U_{i}, \quad  \| \theta \| = \| \eta \| = 1,
%\quad \theta\in \RR^{k}, \eta \in \RR^{k+1}$; 
%\[
%\max_{\theta,\eta}  \| v \|, \; \text{ where } \; v =
%\tmr[1,3]{\cA}{U_{i+1} \theta,W_{i+1} \eta}  
%\]
%and such that $v \perp V_{i}, \quad  \| \theta \| = \| \eta \| = 1,
%\quad \theta,\eta \in \RR^{k+1}$. 

Recall the assumption  that $\cA \in \RR^{l \times m \times n}$ is large and
sparse.  Clearly the optimization approach has the drawback that the tensor
$\cC_w$ is generally a dense tensor of dimension $i \times i \times n$, and
the computation of the best rank-$(1,1,1)$ approximation or the HOSVD of
that tensor can be quite time-consuming.  Of course, in an application,
where it is essential to have a good approximation of the tensor with as
small dimensions of the subspaces as possible, it may be worth the extra
computation needed for the optimization. However, we can avoid handling
large, dense tensors by modifying  the optimized recursion, so that an
approximation of the solution of the maximization problem \eqref{eq:Copt}
is computed using 
% OK \todo[color=green]{Jag byter $p$ till $t$, $p,q,r$ anvands som rangen
% av approximationen.}
$t$ steps of the minimal Krylov recursion on the tensor $\cC_w$, for small
$t$.

Assume that we have computed a rank-$(t,t,t)$ approximation of $\cC_w$, 
\[
\cC_w \approx \tml{\Theta,H,\Omega}{\cS_w},
\]
for some small value of $t$, using the minimal Krylov method.  By computing
the best rank-$(1,1,1)$ (or HOSVD) approximation of the small tensor
$\cS_w \in \RR^{t \times t \times t}$, we obtain an approximation of the solution of \eqref{eq:Copt}. It
remains to demonstrate that we can apply the minimal Krylov recursion to
$\cC_w$ without forming that tensor  explicitly. Consider the computation
of a vector $\omega$ in the third mode, given the vectors $\theta$, and
$\eta$:  
\begin{align}
 \widehat{ \omega} &= \tmr[1,2]{\cC_w}{\theta,\eta} = 
\tmr[1,2]{\left(\tmr{A}{U_{i},V_{i},I-W_{i-1}W_{i-1}\tp}\right)}{\theta,\eta} \label{eq:opt-w} \\
&= \tmr[3]{\left(\tmr[1,2]{\cA}{U_i \theta, V_i \eta}\right)}{I-W_i W_i\tp} = 
(I-W_i W_i\tp) \tilde{\omega}. \notag
\end{align}
Note that the last matrix-vector multiplication is equivalent to the
Gram-Schmidt orthogonalization in the minimal Krylov algorithm.  Thus,
we have only a sparse tensor-vector-vector operation, and a few 
matrix-vector multiplications, and similarly for the computation of
$\widehat{\theta}$ and $\widehat{\eta}$. 

It is crucial for the performance of this outer-inner Krylov procedure that
a good enough approximation of the solution of \eqref{eq:Copt} is obtained
for small $t$, e.g. $t$ equal to 2  or 3. We will see in our numerical examples
that it gives almost as good results as the implementation of the full
optimization procedure.  

\subsection{``Small'' Mode}
\label{sec:kryl-small}

In information science applications it often happens that one of the
tensor modes has much smaller dimension than the others. For
illustration assume that the first mode is small, i.e. $l \ll
\min(m,n)$.  Then in the Krylov variants described so far, after $l$
steps the algorithm has produced a full basis in that mode, and no
more need be generated.  Then the question arises which $u$-vector to
choose, when new basis vectors are generated in the other modes. Two
obvious alternatives are to use the vectors $u_1, \ldots, u_l$ in a
cyclical way, or to take a random linear combination.  
One may also apply the optimization idea in that mode, i.e. in the
computation of $w_{i}$ perform the maximization
\[
\max_{\theta}  \| \widehat{w} \|, \; \text{ where } \;  \widehat{w} =
\tmr[1,2]{\cA}{U_i \theta,v_i},  \quad 
  \widehat{w} \perp W_{i-1}, \quad  \|
\theta \| =  1, \quad \theta \in \RR^{i}. \nonumber
\]
The problem can be solved by computing a best rank-1 approximation of
the matrix 
\begin{equation*}
C_{w} = \tmr{A}{U_{i},v_{i},I-W_{i-1}W_{i-1}\tp}. 
\end{equation*}
As before, this is generally a dense matrix with one large mode. A
rank one approximation can again be computed, without forming the
dense matrix explicitly, using a Krylov method (here the Arnoldi
method).

\subsection{Krylov Subspaces for Contracted Tensor Products}
\label{sec:kryl-contr}

Recall from Section \ref{sec:GKB} that the Golub-Kahan bidiagonalization procedure 
generated matrices $U_k,V_k$, which are orthonormal basis vectors for the Krylov
subspaces of $A A\tp$ and $A\tp A$, respectively. In tensor notation
those products may be written as
\[
\ctp[-1]{A,A} = A A\tp, \qquad \qquad \ctp[-2]{A,A} = A\tp A.
\]
For a third order tensor $\cA \in \RR^{l \x m \x n }$, and starting
vectors $u \in \RR^l, v \in \RR^m, w \in \RR^n$  we may consider the
matrix Krylov subspaces
\begin{align*}
& \cK_p(\ctp[-1]{\cA,\cA},u), & \ctp[-1]{\cA,\cA} & = A^{(1)} (
A^{(1)})\tp \in \RR^{l \x l},\\
& \cK_q(\ctp[-2]{\cA,\cA},v), & \ctp[-2]{\cA,\cA} & = A^{(2)} (
A^{(2)})\tp \in \RR^{m \x m}, \\
& \cK_r(\ctp[-3]{\cA,\cA},w), & \ctp[-3]{\cA,\cA} & = A^{(3)} (
A^{(3)})\tp \in \RR^{n \x n}.
\end{align*}
The expressions to the right in each equation are matricized tensors. It suffices for our discussion to know that for an $l \x m \x n$ tensor $\cA$ one can associate three matrices $A^{(1)} \in \RR^{l \x mn}$, $A^{(2)} \in \RR^{m \x ln}$ and $A^{(3)} \in \RR^{n \x lm}$. For details the interested reader may consider \cite{elden09,bader07,latha00,eccv:02}. In this case we reduce a third order tensor to three different
(symmetric) matrices, for  which we compute the usual matrix subspaces through the
Lanczos recurrence.  This can be done without explicitly computing the products $\ctp[-i]{\cA,\cA}$, thus taking advantage of sparsity. To illustrate this consider the matrix times vector operation $A^{(1)}(A^{(1)})\tp u$, which can be written 
\begin{equation}
\label{eq:ctpVec}
[A_{1} \; \dots \;A_{n}][A_{1} \; \dots \;A_{n}]\tp u = \sum_{i=1}^{n} A_{i} A_{i}\tp u,
\end{equation}
where $A_{i} = \cA(:,:,i)$ is the i'th frontal slice of $\cA$. 

The result of the Lanczos process separately on the three contracted tensor products is three sets of orthonormal basis vectors for each of the
modes of the tensor, collected in $U_p, V_q, W_r$, say. A low-rank
approximation of the tensor can then be obtained using Lemma
\ref{lem:LS}.

It is straightforward to show that if $\cA = \tml{X,Y,Z}{\cC}$ with $\rank(\cA) = (p,q,r)$, then the contracted tensor products  
\begin{align}
\ctp[-1]{\cA,\cA} & = A^{(1)} (A^{(1)})\tp = X C^{(1)}(Y \otimes Z)\tp( Y \otimes Z) (C^{(1)})\tp X\tp, \\
\ctp[-2]{\cA,\cA} & = A^{(2)} (A^{(2)})\tp= Y C^{(2)}(X \otimes Z)\tp (X \otimes Z) (C^{(2)})\tp Y\tp, \\
\ctp[-3]{\cA,\cA} & = A^{(3)} (A^{(3)})\tp = Z C^{(3)}(X \otimes Y)(X \otimes Y) (C^{(3)})\tp Z\tp,
\end{align} 
are matrices with ranks $p$, $q$ and $r$, respectively. Then it is clear that the separate Lanczos recurrences will generate matrices $U,V,W$ that span the same subspaces as $X,Y,Z$ in $p$, $q$ and $r$ iterations, respectively.  

\paragraph{Remark}  Computing $p$ (or $q$ or $r$) dominant
eigenvectors of the symmetric positive semidefinite matrices
$\ctp[-1]{\cA,\cA},\ctp[-2]{\cA,\cA}$, $\ctp[-3]{\cA,\cA}$,
respectively, is equivalent to computing the truncated  HOSVD of $\cA$. We
will show the calculations for the first mode. Using the HOSVD $\cA =
\tml{U,V,W}{\cS}$, where now $U$, $V$, and $W$ are orthogonal matrices and the
core  $\cS$ is all-orthogonal \cite{latha00}, we have  
\[
\ctp[-1]{\cA,\cA} = U S^{(1)}(V \otimes W)\tp( V \otimes W) (S^
{(1)})\tp U\tp = U \bar{S} U\tp,
\]
where $\bar{S} = S^{(1)} (S^{(1)})\tp = \diag(\s_{1}^{2},\s_{2}^{2},\dots,\s_{l}^{2})$ with $\sigma_{i}^{2} \geq \sigma_{i+1}^{2}$ and $\s_{i}$ are first mode multilinear singular values of $\cA$.

\subsection{ Complexity}
In this subsection we will discuss the amount of computations
associated to the different methods.  Assuming that the tensor is
large and sparse it is likely  that,
for small values of $k$ (compared to $l$, $m$, and $n$), the dominating work in computing a rank-$(k,k,k)$
approximation is due to tensor-vector-vector multiplications. 
%Lower computational complexity due to eventual sparsity in the tensor will
%be accounted for implicitly in this product. 
%\todo[color=green]{\textcolor{red}{Jag la 
%till en mening.} Jag tycker inte meningen tillförde något.} 
\paragraph{Minimal Krylov recursion}
Considering Equation \eqref{eq:low-rank}, it is clear that computing
the $k\x k \x k$ core tensor $\cH$ is necessary to have a low rank
approximation of $\cA$.  From the proof of Proposition
\ref{prop:minK-fact} we see that $\cH$ has a certain Hessenberg
structure along and close to its ``two-mode diagonals''. However, away
from the ``diagonals'' there will be no systematic structure.
% \todo[color=green]{\textcolor{red}{Hessenberg strukturen ar oberoende...} Vi
% genererar vector som är associerade med diagonalen. Borta från diagonalen
% får vi ingen Hessenbergstruktur.} 
We
can estimate that the total number of tensor-vector-vector
multiplications for computing the $k \times k \times k$ tensor $\cH$
is $k^2$. The computation of $\cH$ can be split as
\[
\cH = \tmr{\cA}{U_{k},V_{k},W_{k}} = \tmr[3]{\cA_{uv} }{W_{k}}, \quad
\text{ where } \quad  
\cA_{uv} = \tmr[1,2]{\cA}{U_{k},V_{k}}. 
\] 
There are $k^{2}$ tensor-vector-vector products for computing the $k\x
k \x n$ tensor $\cA_{uv}$. The complexity of the following computation
$\tmr[3]{\cA_{uv}}{W_{k}}$ is $O(k^{3}n)$, i.e. about $k^{3}$
vector-vector inner products.   

Several of the elements of the core tensor are available from the
generation of the Krylov vectors. Naturally they should be saved to
avoid unnecessary work. Therefore we need not include the $3k$
tensor-vector-vector multiplications from the recursion in the
complexity. 

\paragraph{Maximal Krylov Recursion}
There are several options in the use of the maximal recursion for
computing a rank-$(k,k,k)$ approximation. One may apply the method
until all subspaces have dimension larger than $k$. In view of the
combinatorial complexity of the method the number of
tensor-vector-vector multiplications can then be much higher than in
the minimal Krylov recursion. Alternatively, as soon as one of the
subspaces reaches dimension $k$, one may stop the maximal recursion
and generate only the remaining vectors in the other two modes, so
that the final rank becomes $(k,k,k)$. That variant has about the same
complexity as the minimal Krylov recursion.

\paragraph{Optimized Krylov Recursion}
The optimized recursion can be implemented in different ways. 
 In Section \ref{sec:kryl-opt} we described a variant based on ``inner Krylov
 steps''. Assuming that we perform $t$ 
% OK \todo[color=green]{Jag bytte    $q$ till $t$}
inner Krylov steps, finding  the (almost) optimal
  $\wh{w}$  \eqref{eq:opt-w} requires $3t$ tensor-vector-vector
 multiplications. Since the optimization is done in $k$ outer
 Krylov steps in three modes we perform  $9kt$ such multiplications. 
%\todo[color=green]{\textcolor{red}{Jag fortydligade lite mer.} Och jag
%lite till. }
The total complexity becomes $k^2 + 9 k t$.  In \cite{gos10}
 another variant is described where the optimization is done on the core
 tensor. 

\paragraph{``Small'' Modes}
Assume that the tensor has one small mode  and that
a random or fixed combination of vectors is chosen in this mode when
new vectors are generated in the other modes. Then the complexity
becomes $k^2+2k$. 

\paragraph{Krylov Subspaces for Contracted Tensor Products}
In each step of the Krylov method a vector is
multiplied  by a contracted tensor product. This can be implemented using 
\eqref{eq:ctpVec}. If we assume that each such operation has the same
complexity as two tensor-vector-vector
% OK   \todo[color=green]{Jag la till ``-vector'' i rott.} 
multiplications,  then the complexity becomes 
$k^2 + 6 k$, where the second order term is for computing the core
tensor.
%Working with a dense $l \x m \x n$ tensor $\cA$ one can verify that a
%single product $\ctp[-1]{\cA,\cA} u$ is computed in about $4mnl$
%floating point operations. The complexity is the same for both of
%$\ctp[-2]{\cA,\cA} v$ and $\ctp[-3]{\cA,\cA} w$. Thus computing a
%rank-$(p,q,r)$ approximation would require about $4(p+q+r)mnl$ flops.
%The cost for the Lanczos processes is of lower order of complexity. In
%addition we may want to compute the core tensor as well, ....

The complexities for four of the methods are summarized in Table~\ref{tab:complexity}.
\begin{table}[htdp!]
\begin{center}
\begin{tabular}{lc}
\hline \hline 
Method & Complexity \\
\hline \hline 
Minimal Krylov  & $k^2$ \\
%Maximal Krylov &  \\
Optimized minimal Krylov & $k^2 + 9 k t$\\
``Small'' mode & $k^2+2k$ \\
Contracted tensor products & $k^2 + 6 k$  \\
\hline \hline 
\end{tabular}
\end{center}
\caption{Computational complexity   (tensor-vector-vector
  multiplications) for the computation of a rank-$(k,k,k)$
  approximation with  different methods. In the optimized Krylov
  recursion $t$ inner Krylov steps are made.} 
\label{tab:complexity}
\end{table}%

\section{ Numerical Examples}
\label{sec:numerical}
The purpose of the examples in this section is to make a preliminary
investigation of the usefulness of the concepts proposed. We will
generate matrices $U,V,W$ using the various Krylov procedures and, in
some examples for
comparison,  the truncated HOSVD. Given a tensor $\cA$ and
matrices $U,V,W$ the approximating tensor $\tilde{\cA}$ has the form  
\begin{equation}
\label{eq:app}
\tilde{\cA} = \tml{U,V,W}{\cC}, \quad \text{ where } \quad \cC = \tmr{\cA}{U,V,W}.
\end{equation}
Of course, for large problems computing $\tilde{\cA}$ explicitly (by
multiplying together the matrices and the core $\cC$) will
not be feasible, since that tensor will be dense.  However, it is easy to show that approximation error
is
%\todo[color=green]{Detta är mycket bättre!} 
\[
\| \cA - \tilde{\cA} \|^2 = \| \cA \|^2 - \| \cC \|^2.
\]

%But we will always have $U,V,W$ and $\cC$ at hand and the
%goodness of an approximation is determined by the subspaces of $U,V,W$
%and will be measured by $\| \cC \|$. The larger $\| \cC \|$ the better
%is the corresponding approximation. It is also easy to show that for
%orthonormal matrices $U,V,W$ we have the bound $\| \cC \| \leq \|\cA
%\|$.  

%We will also compare different subspaces corresponding to the same
%mode of the tensor and for this we will compute the principal angles
%between two subspaces \cite{bjgo:73}. For orthonormal matrices with
%the same number of rows $U$ and $\bar{U}$, the principal angles
%$\theta_{i}$ are given by $\theta_{i} = \arccos(\sigma_{i})$ where
%$\sigma_{i}$ are the singular values of  $U\tp \bar{U}$. 

%\todo[color=green]{Stycket omskrivet, smidigare formulerat. Men vad är
%relevansen här?} 
For many applications a low rank
approximation is only an intermediate or auxiliary result, see e.g.
\cite{savas07}.  It sometimes holds that the better the approximation
(in norm),  the 
better it will perform in the particular application. But quite often,
especially in information science applications, good performance is
obtained using an approximation with quite high error, see
e.g. \cite{jema01}. Our experiments will focus on how good
approximations are obtained by the proposed methods. How these low
rank approximations are used will depend on the application as well as
on the particular data set.  
%\todo[color=red]{Jag la till tva meningar i slutet.}
%approximation to be particularly good in order to perform reasonably
%good in various tasks. This is especially true for data originating
%from  which are often large and very
%sparse \cite{elden07}.  

For the timing experiments in Sections  \ref{sec:testMin} and \ref{sec:testDigits}
 we used a MacBook laptop with 2.4GHz processor
and 4 GB of main memory. For the experiments on the Netflix data in
Section \ref{sec:testNetflix} we used a 64 bit Linux machine with 32
GB of main memory running Ubuntu. The calculations were performed
using Matlab and the TensorToolbox, which 
 supports computation with sparse tensors 
\cite{bader06b, bader07}.

\subsection{Minimal Krylov Procedures}
\label{sec:testMin}
We first made a set of experiments to confirme that  the result
in Theorem \ref{thm:minKrylov22}  holds for a numerical
implementation, using synthetic data generated with a specified low
rank.

It is not uncommon that tensors originating from signal processing
applications have low multilinear ranks. Computing the HOSVD of such a
tensor $\cA$ can be done by direct application of the SVD on the
different matricizations  $A^{(i)}$ for $i = 1,2,3$. An alternative is
to first compute $U_{p},V_{q},W_{r}$ using the modified minimal Krylov
procedure. Then we have the decomposition $\cA =
\tml{U_{p},V_{q},W_{r}}{\cH}$. To obtain the HOSVD of $\cA$ we compute
the HOSVD of the much smaller\footnote{$\cA$ is a $l \x m \x n$ tensor
  and $\cH$ is a $p \x q \x r$, and usually the multilinear ranks
  $p,q,r$ are much smaller than the dimensions $l,m,n$ of $\cA$.}
tensor $\cH = \tml{\bar{U},\bar{V},\bar{W}}{\cC}$. It follows that the
singular matrices for $\cA$ are given by $U_{p}\bar{U}$,
$V_{q}\bar{V}$ and $W_{r}\bar{W}$. We conducted a few experiments to
compare timings for the two approaches. 
%We also performed a timing comparison between the minimal Krylov
%recursion and the truncated HOSVD.
Tensors with three
different dimensions were generated and for each case we used three different low
ranks. The tensor dimensions, their ranks and the computational times
for the respective case are presented in Table \ref{tab:timings}.  
\begin{table}[t!]
\begin{center}
\begin{tabular}{c|cc|cc|cc}
\hline \hline
$\dim(\cA)$\textbackslash$\rank(\cA)$ & \multicolumn{2}{c|}{$(10,10,10)$} & \multicolumn{2}{c}{$(10,15,20)$} & \multicolumn{2}{|c}{$(20,30,40)$} \\
\hline 
Method & (1)&(2) &(1)&(2)&(1)&(2)\\
\hline \hline
$50 \x 70 \x 60$ & 0.087 & 0.165 & 0.113  & 0.162& 0.226& 0.169\\
$100 \x 100 \x 100$ & 0.38 & 2.70 & 0.44 & 2.72&0.91 & 2.71 \\
$150 \x 180 \x 130$ & 1.32 & 11.27 & 1.44 & 11.07 & 3.01 & 11.01 \\
%$50 \x 70 \x 60$ & 0.087 --- 0.089--- 0.165 & 0.113---0.118---0.162& 0.226---0.250---0.169\\
%$100 \x 100 \x 100$ & 0.38---0.38---2.70 & 0.44---0.45---2.72&0.91---0.93---2.71 \\
%$150 \x 180 \x 130$ & 1.32---1.33---11.27 & 1.44---1.45---11.07 & 3.01---3.03---11.01 \\
\hline \hline
\end{tabular}
\end{center}
\caption{Timing in seconds for computing the HOSVD of low rank tensors
  using (1)  the
  modified minimal Krylov method and HOSVD of the smaller core $\cH$ and
  (2) truncated HOSVD approximation of $\cA$.} 
\label{tab:timings}
\end{table}%
%To compute the HOSVD of $\cA$ from the Krylov approximation we would
%need to compute the HOSVD of the small $p \x q \x r$ tensor
%$\tml{U_{p},V_{q},W_{r}}{\cA}$ and then change the basis
%representation. The total additional timing required is in the range
%0.002--0.05 seconds for all cases. 
We see that for the larger problems the computational time
for the HOSVD is 2--8
times longer than for  the modified minimal Krylov procedure with
HOSVD on the core tensor.  Of
course, timings of Matlab codes are unreliable in general, since the
efficiency of execution depends on how much of the algorithm is
user-coded and how much is implemented in Matlab low-level functions
(e.g. LAPACK-based). It should be noted that the tensors in this
experiment  are dense, and much of the HOSVD computations are done in
low-level functions. Therefore, we believe that the timings are
rather realistic.

%Original tensor is of size 50 100 80
%test using k = 5,10,20
%these are means for the first mode with p = 11

Next we   let $\cA \in \RR^{50 \x 60 \x 40}$ be a random tensor, and 
computed a rank-$(10,10,10)$ approximation using the minimal Krylov
recursion and a different approximation using truncated HOSVD. Let the optimal cores computed using Lemma
\ref{lem:LS} be  denoted $\cH_{\text{min}}$ and  $\cH_{\text{hosvd}}$,
respectively. We made this calculation for 100
different random tensors and report $(\|\cH_{\text{min}} \| -
\|\cH_{\text{hosvd}} \|)/\|\cH_{\text{hosvd}} \|$ for each
case. Figure \ref{fig:test1}  illustrates the outcome. Clearly, if the
relative difference is larger than 0, then the Krylov method gives a
better approximation. 
\begin{figure}[tbp!]    
\centering
\includegraphics[width=.8\textwidth]{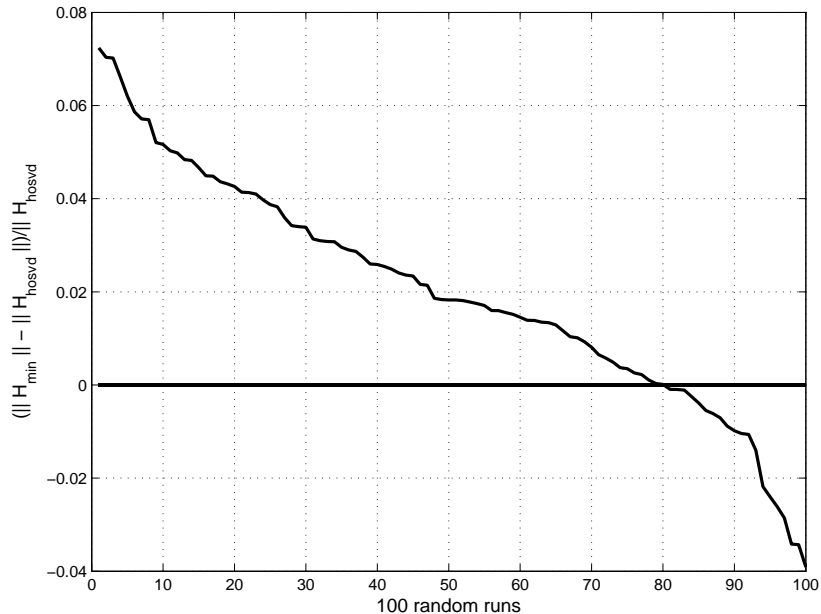}
\caption{Difference between $\|\cA_{\text{min}}
  \|$, approximation obtained with the minimal Krylov method,  and
  $\|\cA_{\text{hosvd}} \|$, approximation obtained by the truncated
  HOSVD of a $50 \x 60 \x 40$ tensor $\cA$. The rank of the
  approximations were $(10,10,10)$.  }
\label{fig:test1}
\end{figure}
In about 80\% of the runs the minimal Krylov method generated better
approximations than the truncated HOSVD, but the difference was quite
small.
% Nästa stycke behövs inte eftersom vi har relativ jämförelse.
% Over the 100 different
%runs, we observed that mean$(\|\cA\|) = 346.36$,
%mean$(\|\cA_{\text{min}}\|) = 43.44$ and mean$(\|\cA_{\text{hosvd}}\|)
%= 42.89$.  
      
%Let now $\cA$ represent the potential function $f(x)$ above. In this
%case we obtain almost the same subspace as the HOSVD, the relative
%error in the approximation for the minimal Krylov is about $1.3\cdot
%10^{-3}$ while for the HOSVD solution is $1.2\cdot 10^{-5}$. For this
%tensor the multilinear singular values decrease extremely fast, only
%about 20 of them are above the machine precision level. 

In the following experiment we compared the performance of different
variants of the optimized minimal Krylov recursion applied to sparse
tensors. We generated  
tensors based on Facebook graphs for different US universities
\cite{tkmp08}. The Caltech graph is represented by a $597 \times 597$
sparse matrix. For each individual there is housing information. Using
this we generated a tensor of dimension $597 \times 597 \times 64$,
with 25646 nonzeros. The purpose was to see how good approximations
the different methods gave as a function of the subspace dimension. 
We compared the minimal Krylov recursion to the following optimized
variants:  

\begin{description}
\item {\textbf{Opt-HOSVD.}} The minimal Krylov recursion with optimization
  based on HOSVD of the core tensor  \eqref{eq:Cw}. This variant is
  very costly and is included only as a benchmark. 

\item {\textbf{Opt-Krylov.}} The minimal Krylov recursion that utilized three inner Krylov steps to obtain approximations to the optimized linear combinations. This is an implementation of the discussion from the second part of Section  \ref{sec:kryl-opt}. 

\item {\textbf{Opt-Alg8.}}  Algorithm 8 in \cite{gos10}\footnote{The
    algorithm involves the approximation of the dominant singular
    vectors of a matrix computed from the core tensor. In \cite{gos10}
  the power method was used for this computation. We used a linear
  combination of the first three singular vectors of the matrix,
  weighted by the singular values. }. 

\item {\textbf{Truncated HOSVD.}} This was included as a benchmark
  comparison. 

\item {\textbf{minK-back.}} In this method we used the minimal Krylov method but performed 
  10 extra steps. Then we formed the core $\cH$ and computed a truncated HOSVD
  approximation of $\cH$. As a last step we truncated the Krylov subspaces
  accordingly. 
\end{description}

In all Krylov-based methods we used four initial minimal Krylov steps
before we started using the optimizations. 

\begin{figure}[tbp!]
\centering
\includegraphics[width=.8\textwidth]{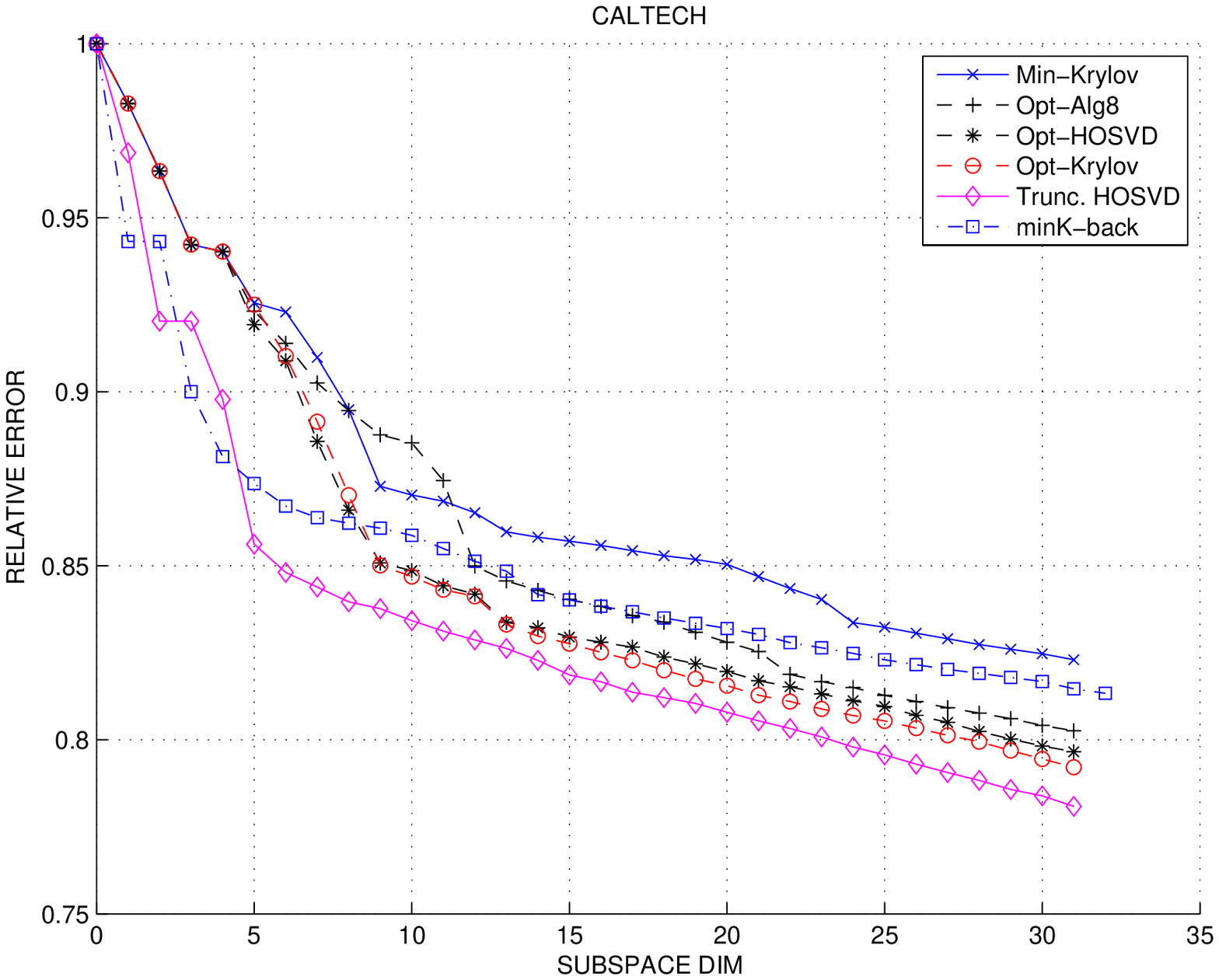}\\
\includegraphics[width=.8\textwidth]{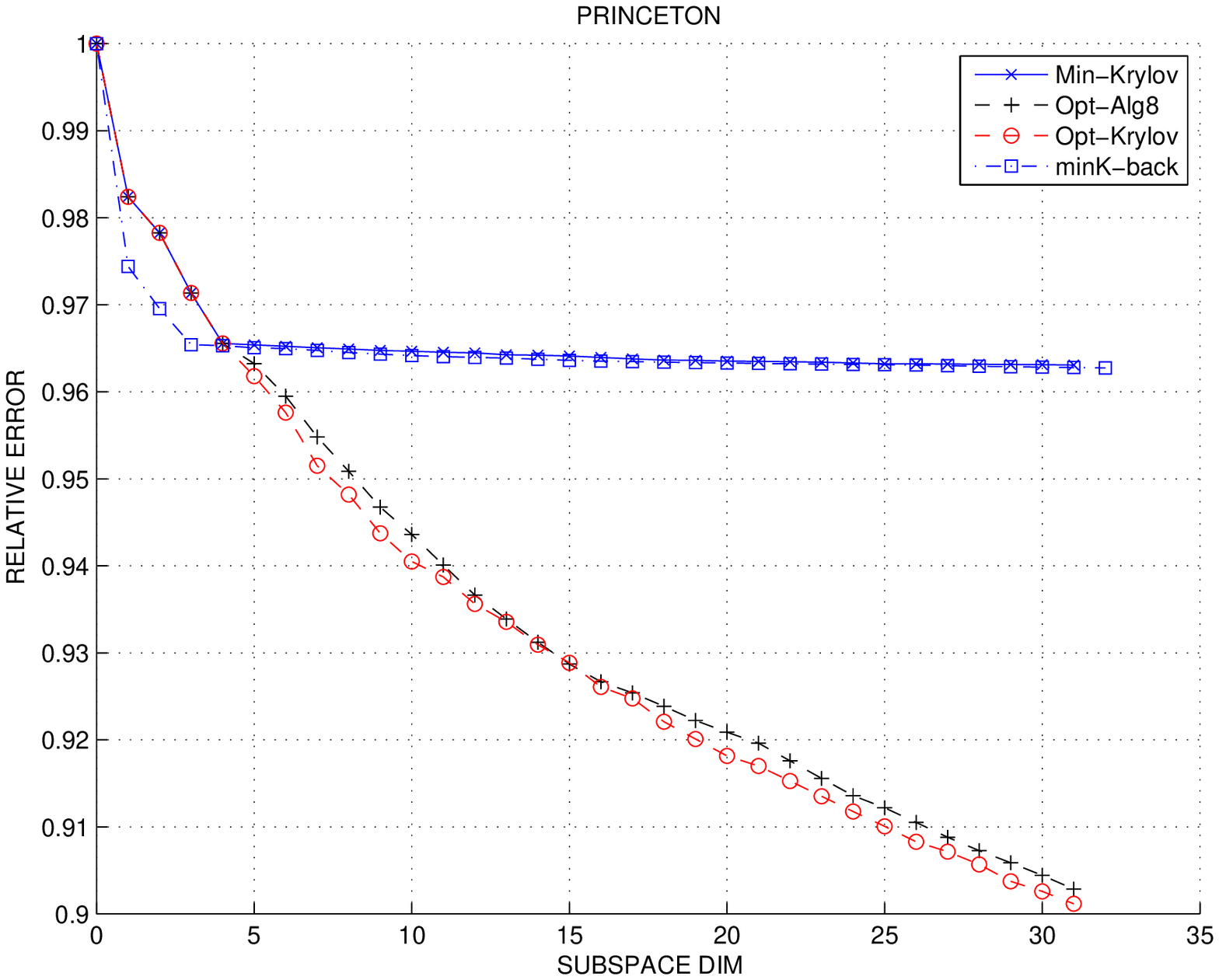}
\caption{Errors in the low rank approximations of the sparse Caltech (top)
  and Princeton (bottom) tensors.}
\label{fig:optimcaltech}
\end{figure}

Another sparse tensor was created using the Facebook data from
Princeton. Here the tensor was constructed using a student/faculty
flag as third mode, giving a $6593 \x 6593 \x 29$ tensor with 585,754
non-zeros. 

The results are illustrated in Figure \ref{fig:optimcaltech}. We see
that for the Caltech tensor  the ``backward-looking'' variant
(minK-back) gives good approximations for small dimensions as long as
there is a significant improvement in each step of the minimal Krylov
recursion.  After some ten steps all optimized variants give approximations
that are rather close to that of the HOSVD. 

For the Princeton  tensor we only ran the minimal Krylov recursion and
two of the optimizations. Here the optimized versions continued to give
significant improvements as the dimension is increased, in spite of
the poor performance of the minimal Krylov procedure itself. 

\subsection{Test on Handwritten Digits}
\label{sec:testDigits}
%\todo[color=green]{Sektion 5.1 är i imperfekt, här har vi presens och
%  imperfekt blandat.}
%\todo[color=red]{Jag har gatt igenom detta en gang. Kan du kolla sa
%att jag inte missat nagot?} 
Tensor methods for the classification of handwritten digits are
described in \cite{savas03,savas07}. 
We have performed tests using the handwritten digits from the US
postal service database. Digits from the database were formed into a
tensor $\cD$ of dimensions $400 \x 1194 \x 10$. The first mode of the
tensor represents pixels\footnote{Each digit is smoothed and reshaped
  to a vector.}, the second mode represents the variation within the
different classes and the third mode represents the different
classes. Our goal was to find low dimensional subspaces $U_p$ and $V_q$ in
the first and second mode, respectively. The approximation of the original tensor can be written as 
\begin{equation}
\label{eq:digTenApp}
\RR^{400 \x 1194 \x 10}\ni \cD \approx \tml[1,2]{U_p,V_q}{\cF} \equiv \cD_{p,q,10}.
\end{equation}
An important difference compared to  the previous sections
is that here we wanted to find only two of three matrices. The class mode of
the tensor was not reduced to lower rank, i.e. we were computing a
rank-$(p,q,10)$ approximation of $\cD$. We computed  low rank
approximations for this tensor using five different methods: (1)
truncated HOSVD; (2) modified minimal Krylov recursion; (3) contracted
tensor product Krylov recursion; (4) maximal Krylov recursion; and (5)
optimized minimal Krylov recursion. Figure \ref{fig:digits} shows the
obtained results for low rank approximations with $(p,q) = \{
(5,10),\; (10,20),\; (15,30),\; \cdots \;, \;(50,100)\}$. The reduction of dimensionality in two modes required
special treatment for several of the methods. We will describe each
case separately.  
In each case the low rank approximation $\cD_{p,q,10}$ is given by 
\begin{equation}
\label{eq:digApp}
\cD_{p,q,10} = \tml[1,2]{U_{p}U_{p}\tp,V_{q} V_{q}\tp}{\cD}
\end{equation}
where the matrices $U_{p}$ and $V_{q}$ were obtained using different methods.
\paragraph{Truncated HOSVD} We computed the HOSVD $\cD = \tml{U,V,W}{\cC}$ and truncated the multilinear singular matrices, i.e. $U_{p} = U(:,1:p)$ and $V_{q} = V(:,1:q)$.
\paragraph{Modified minimal Krylov recursion} The minimal Krylov recursion was modified in several respects. We ran Algorithm \ref{alg:minKrylov} for 10 iterations and obtained $U_{10},V_{10},W_{10}$. Next we ran $p-10$ iterations and generated only $u$ and $v$ vectors. For every new $u_{k+1}$ we used $\bar{v}$ and $\bar{w}$ as random liner combination of vectors in $V_{k}$ and $W_{10}$, respectively. In the last $q-p$ iterations we only generated new $v$ vectors using again random linear combinations of vectors in $U_{p}$ and $W_{10}$.
\paragraph{Contracted tensor product Krylov recursion} For this method
we applied $p$ and $q$ steps of the Lanczos method with random starting vectors on the
matrices $\ctp[-1]{\cA,\cA}$ and $\ctp[-2]{\cA,\cA}$, respectively.
% Do observe that it is possible to compute dominant eigenspaces of
% the two respective contracted products, but that would give the same
% solution as the truncated HOSVD. 
\paragraph{Maximal Krylov recursion} We used the maximal Krylov recursion with starting vectors $u_{1}$ and $v_{1}$ to generate $W_{1} \rightarrow U_{2} \rightarrow V_{3} \rightarrow W_{6} \rightarrow U_{19} \rightarrow V_{115}$. Clearly, we now need to make modification to the general procedure. $W_{10}$ was obtained as the third mode singular matrix from the HOSVD of the product $\tmr[1,2]{\cD}{U_{19},V_{115}}$ and as a last step we computed $U_{100}$ as the 100 dimensional dominant subspace obtained from the first mode singular matrix of $\tmr[2,3]{\cD}{V_{114},W_{10}}$. The $U_{p}$ and $V_{q}$ used for the approximation in \eqref{eq:digApp} were constructed as follows: We formed the product $\bar{\cC} = \tmr[1,2]{\cD}{U_{100},V_{115}}$ and computed the HOSVD $\tml{\tilde{U},\tilde{V},\tilde{W}}{\tilde{\cC}} = \bar{\cC}$. Then $U_{p} = U_{100}\tilde{U}(:,1:p)$ and $V_{q} = U_{115}\tilde{V}(:,1:q)$.
\paragraph{Optimized minimal Krylov recursion} For this 
minimal Krylov recursion we used the optimization approach, for every new vector, from the first part of  Section \ref{sec:kryl-opt}. The coefficients for the linear combinations were obtained by best rank-$(1,1,1)$ approximation of factors as $\cC_{w}$ in Equation \eqref{eq:Cw}.
\begin{figure}[tbp!]
\centering
\includegraphics[width=0.8\textwidth]{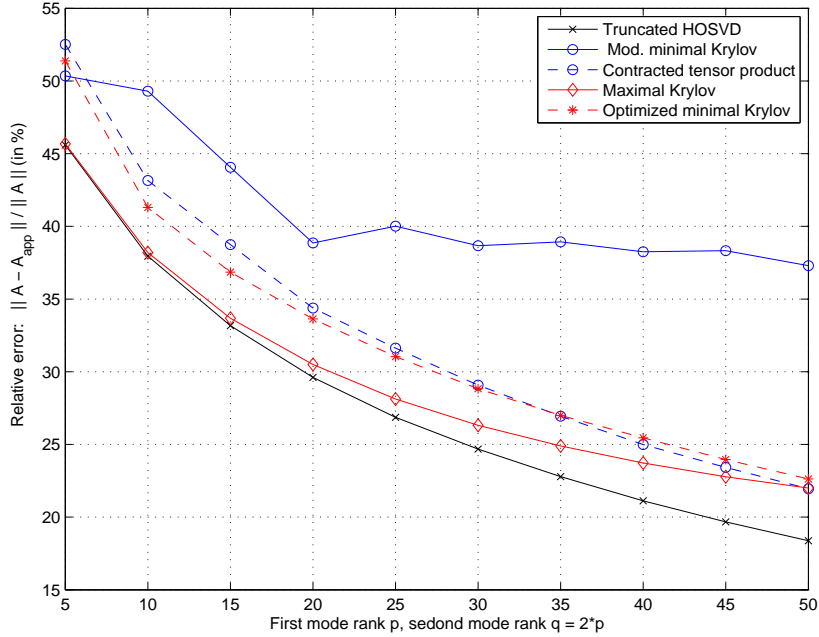}
\caption{The relative error of the low rank approximations obtained using five different methods. The $x$-axis indicates the ranks $(p,q,10) = (p, 2p, 10)$ in the approximation. }
\label{fig:digits}
\end{figure}

The experiments with the handwritten digits are illustrated in Figure
\ref{fig:digits}. 
We made the following observations: (1) The truncated HOSVD
give the best approximation for all cases; (2) The minimal Krylov
approach did not perform well. We  observed several 
breakdowns for this methods as the ranks in the approximation
increased. Every time the process broke down we used a randomly generated vector
that was orthonormalized against all previously generated vectors. (3)
The Lanczos method on the contracted tensor products performed very
similar as the optimized minimal Krylov method. (4)  
The performance of the maximal Krylov method was initially as good as the truncated HOSVD but its performance degraded eventually. 

Classifications results using these subspaces in the algorithmic setting presented in \cite{savas07} are similar for all methods, indicating that all of the methods capture subspaces that can used for classification purposes. 

\subsection{Tests on the Netflix Data}
\label{sec:testNetflix}

A few years ago, the Netflix company announced a
competition\footnote{The competition has obtained huge attention from
  many researcher and non-researchers. The improvement of 10 \% that
  was necessary to claim the prize for the contest was achieved by
  join efforts of a few of the top teams \cite{lohr09}.} 
%\todo[color=green]{Detta ar NY Times artikel. Ta den!} 
to improve their  algorithm 
for movie recommendations. Netflix made available movie ratings from
480,189 users/costumers on 17,770 movies  during a time period of 2243
days. In total there were over 100 million ratings. We will not
address the Netflix problem, but we will use the data to test some
of the Krylov methods we are proposing. For our experiments we formed
the tensor $\cA$ that is $480,189 \x 17,770 \x 2243$ and contains all
the movie ratings. Entries in the tensor for which we do not have any
rating were considered as zeros. We  used the minimal Krylov
recursion and the Lanczos process on the products $\ctp[-1]{\cA,\cA}$,
$\ctp[-2]{\cA,\cA}$ and $\ctp[-3]{\cA,\cA}$ to obtain low rank
approximations of $\cA$. 

In Figure \ref{fig:test2} (left plot) we present the norm of the approximation, i.e. $ \|\cA_{\text{min}} \| = \| \cC_{\text{min}} \|$, where $\cC_{\text{min}} = \tmr{\cA}{U_{k},V_{k},W_{k}}$. We have the same low rank approximation in each mode and the ranks range from $k = 5,10,15,\dots,100$. The plot contains three different runs with random initial vectors in all three modes and a fourth curve that is initiated with the means of the first, second and third mode fibers. Observe that for this size of tensors it is practically impossible to form the approximation $\cA_{\text{min}}  = \tml{U_{k}U_{k}\tp,V_{k} V_{k}\tp, W_{k}W_{k}\tp}{\cA}$ since the approximation will be dense.  But the quantity $\| \cC_{\text{min}} \|$ is computable and indicates the quality of the approximation. Larger $\| \cC_{\text{min}} \|$ means better approximation. In fact for orthonormal matrices $U,V,W$ it holds that $\| \cC_{\text{min}} \|\leq \|\cA \|$. 
\begin{figure}[tbp!]
\centering
\includegraphics[width=.48\textwidth]{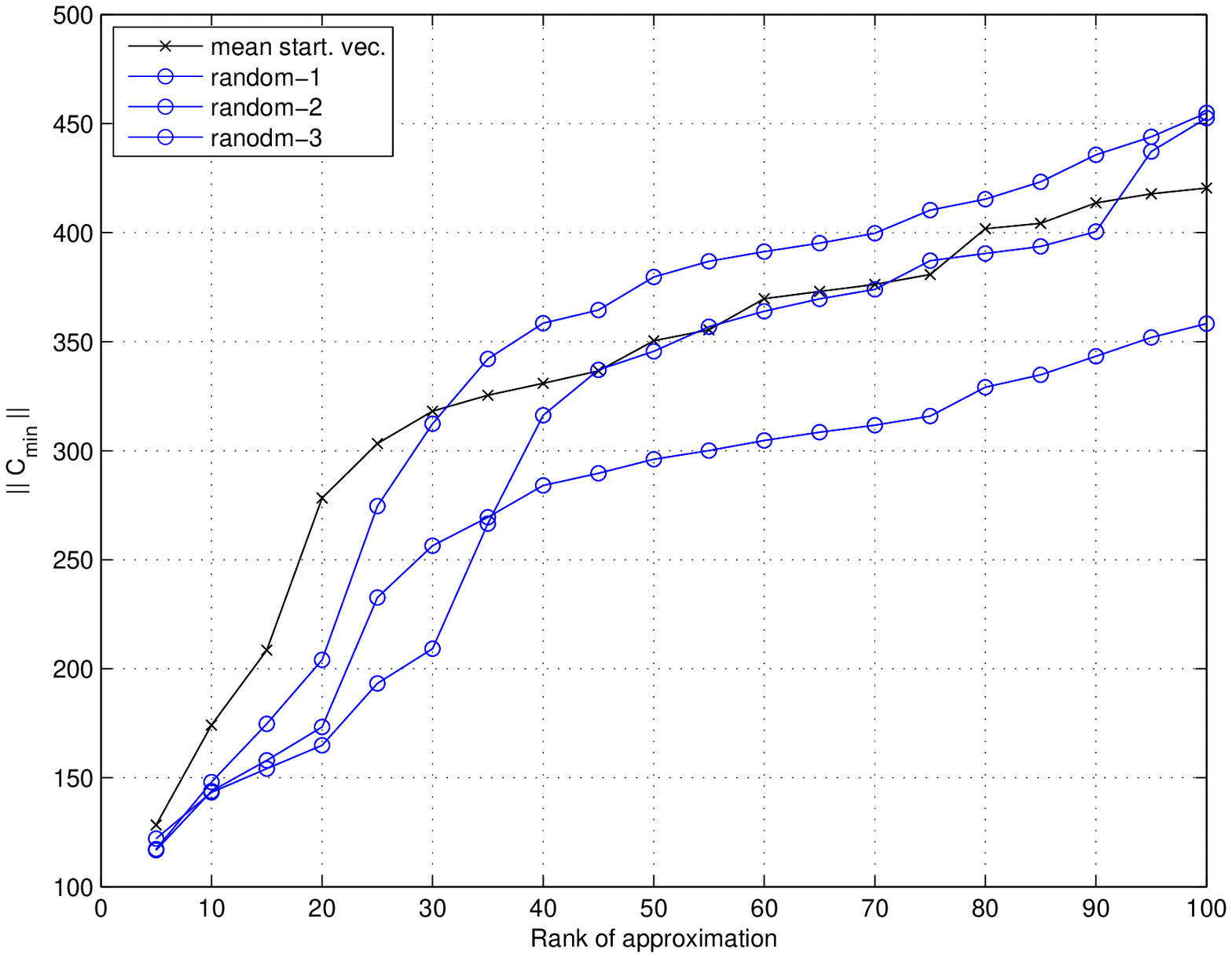}
\includegraphics[width=.48\textwidth]{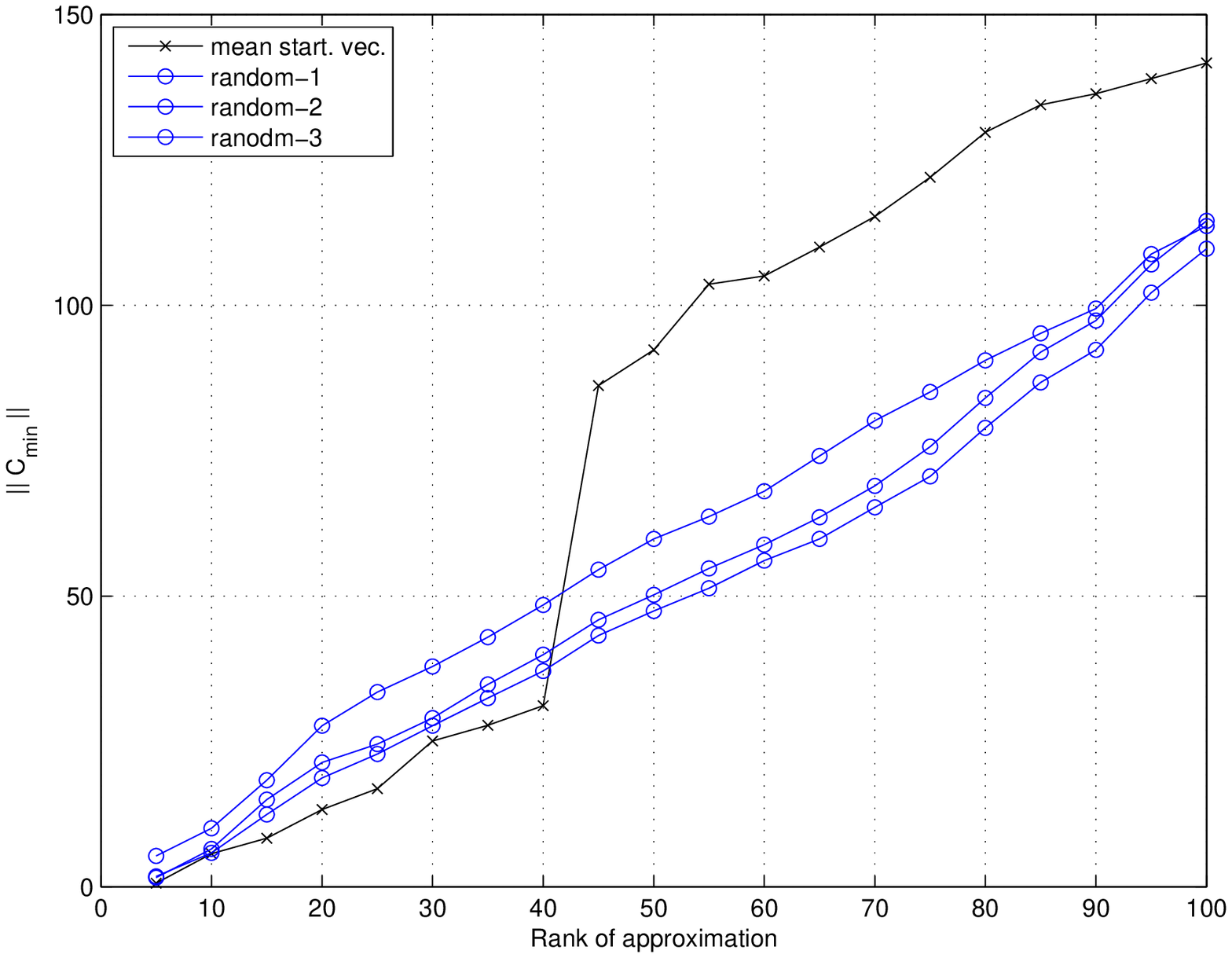}
\caption{We plot $\| \cC_{\text{min}} \|$ as a function of the
    rank $(p,q,r) = (p,p,p)$ in the approximation with $p =
    5,10,15,\dots,100$. \textbf{Left plot:} $U_{p},V_{p},W_{p}$ are
    obtained using the minimal Krylov recursion. Four runs are
    presented: one using starting vectors $u_{1},v_{1},w_{1}$ as the
    means of the mode one, two and three fibers of $\cA$ and three
    runs with different set of random initial vectors. \textbf{Right
      plot:} The subspaces for his case were obtained from separate
    Lanczos recurrences of contracted tensor products. The starting
    vectors were chosen as in the left plot.}
\label{fig:test2}
\end{figure}

Figure \ref{fig:test2} (right plot) contains similar plots, but now the approximating matrices $U_{k},V_{k},W_{k}$ are obtained using the Lanczos process on the symmetric matrices  $\ctp[-1]{\cA,\cA}$, $\ctp[-2]{\cA,\cA}$ and $\ctp[-3]{\cA,\cA}$. We never formed the first two products, but use the computational formula from Equation \eqref{eq:ctpVec} for obtaining first mode vectors and a similar one for obtaining second mode vectors. We did form $\ctp[-3]{\cA,\cA}$ explicitly since it is a relatively small matrix. 
We ran the Lanczos process with ranks $k = 5,10,15,\dots,100$ using random starting vectors in all three modes. Three tests were made and we used the Lanczos vectors in $U_{k},V_{k},W_{k}$. In addition we computed the top 100 eigenvectors for each one of the contracted products. 
  
We  remark that this Netflix tensor is special in the sense that every
third mode fibre, i.e. $\cA(i,j,:)$,  contains only one nonzero
entry.  It follows that the product
$\ctp[-3]{\cA,\cA}$ is a diagonal matrix. Our emphasis for these tests
was to show that the proposed tensor Krylov methods can be employed on
very large and sparse tensors.  

In  this experiment it turned out to be  more efficient to store the
sparse Netflix tensor slice-wise, where each slice itself was a sparse
matrix, than using the sparse tensor format from the TensorToolbox.

\section{Conclusions and Future Work}
\label{sec:conclusion}

In this paper we propose several ways to generalize matrix Krylov methods
to tensors, having applications with sparse tensor in mind. In particular we
 introduce three different methods for tensor approximations. These
are the \textit{minimal}, \textit{maximal Krylov methods} and the
\textit{contracted tensor product methods}. We prove that, given a tensor
of the form $\cA = \tml{X,Y,Z}{\cC}$ with $\rank (\cA) = \rank(\cC) =
(p,q,r)$, a modified version of the minimal Krylov recursion extracts the
associated subspaces of $\cA$ in $\max\{ p,q,r \}$ iterations. We also
investigate a variant of the the optimized minimal Krylov recursion
\cite{gos10}, which gives better approximation than the minimal recursion,
and which can be implemented using only sparse tensor operations.  We 
also show that the maximal Krylov approach generalizes the matrix Krylov
factorization to a corresponding tensor Krylov factorization.

The experiments clearly indicate that the Krylov
methods are useful for low-rank approximation of large and sparse
tensors. In \cite{gos10} it is also shown that they are efficient for
further compression of dense tensors that are given in canonical format. 
The tensor Krylov methods can also be used to speed up HOSVD
computations.  

As the research on tensor Krylov methods is still in a very early stage,
there are numerous questions that need to be answered, and which will
be the subject of our continued research. We have hinted
to some in the text; here we list a few others.

\begin{enumerate}
\item Details with respect to detecting true break down, in floating point arithmetic, and distinguishing those from the case when a complete subspaces is obtained need to be worked out. 
\item A difficulty with Krylov methods for very large problems is that the
  basis vectors generated are in most cases dense. Also, when the required
  subspace dimension is comparatively large, the cost for
  (re)orthogonalization will be high.  For matrices the subspaces can be
  improved using the implicitly restarted Arnoldi (Krylov-Schur) approach
  \cite{arpack98,stew:02}. Preliminary tests indicate that similar
  procedures for tensors may be  efficient. The properties of such methods
  and their implementation will be studied.

\item The efficiency of the different variants of Krylov methods in terms
  of the number of tensor-vector-vector operations, and taking into account the
  convergence rate will be investigated. 

\end{enumerate}

%% main text ====================END=======================================



%% References


%\bibliographystyle{elsarticle-num-names}
\bibliographystyle{elsarticle-harv}
\biboptions{sort&compress}
\bibliography{myBibsEN,general}

\begin{thebibliography}{31}
\expandafter\ifx\csname natexlab\endcsname\relax\def\natexlab#1{#1}\fi
\expandafter\ifx\csname url\endcsname\relax
  \def\url#1{\texttt{#1}}\fi
\expandafter\ifx\csname urlprefix\endcsname\relax\def\urlprefix{URL }\fi

\bibitem[{Bader et~al.(2006)Bader, Harshman, and Kolda}]{bhk:06}
Bader, B.~W., Harshman, R.~A., Kolda, T.~G., 2006. Temporal analysis of social
  networks using three-way {DEDICOM}. Tech. Rep. SAND2006-2161, Sandia National
  Laboratories, Albuquerque, NM.

\bibitem[{Bader et~al.(2007)Bader, Harshman, and Kolda}]{bader07b}
Bader, B.~W., Harshman, R.~A., Kolda, T.~G., October 2007. Temporal analysis of
  semantic graphs using {ASALSAN}. In: ICDM 2007: Proceedings of the 7th IEEE
  International Conference on Data Mining. pp. 33--42.

\bibitem[{Bader and Kolda(2006)}]{bader06b}
Bader, B.~W., Kolda, T.~G., 2006. Algorithm 862: {MATLAB} tensor classes for
  fast algorithm prototyping. ACM Trans. Math. Softw. 32~(4), 635--653.

\bibitem[{Bader and Kolda(2007)}]{bader07}
Bader, B.~W., Kolda, T.~G., 2007. Efficient {MATLAB} computations with sparse
  and factored tensors. SIAM Journal on Scientific Computing 30~(1), 205--231.
\newline\urlprefix\url{http://link.aip.org/link/?SCE/30/205/1}

\bibitem[{Carroll and Chang(1970)}]{cach70}
Carroll, J.~D., Chang, J.~J., 1970. Analysis of individual differences in
  multidimensional scaling via an n-way generalization of {Eckart-Young}
  decomposition. Psychometrika 35, Psychometrika.

\bibitem[{Chew et~al.(2007)Chew, Bader, Kolda, and Abdelali}]{chew07}
Chew, P.~A., Bader, B.~W., Kolda, T.~G., Abdelali, A., 2007. Cross-language
  information retrieval using {PARAFAC2}. In: KDD '07: Proceedings of the 13th
  ACM SIGKDD International Conference on Knowledge Discovery and Data Mining.
  ACM Press, pp. 143--152.

\bibitem[{{De~Lathauwer} et~al.(2000){De~Lathauwer}, Moor, and
  Vandewalle}]{latha00}
{De~Lathauwer}, L., Moor, B.~D., Vandewalle, J., 2000. A multilinear singular
  value decomposition. SIAM J. on Matrix Anal. Appl. 21~(4), 1253--1278.
\newline\urlprefix\url{http://link.aip.org/link/?SML/21/1253/1}

\bibitem[{{De Silva} and Lim(2008)}]{desilva08}
{De Silva}, V., Lim, L.-H., 2008. Tensor rank and the ill-posedness of the best
  low-rank approximation problem. SIAM J. on Matrix Anal. Appl. 30~(3),
  1084--1127.
\newline\urlprefix\url{http://link.aip.org/link/?SML/30/1084/1}

\bibitem[{Eld\'en and Savas(2009)}]{elden09}
Eld\'en, L., Savas, B., 2009. A {Newton--Grassmann} method for computing the
  best multi-linear rank-(${r}_1,r_2,r_3$) approximation of a tensor. SIAM J.
  Matrix Anal. Appl. 31, 248--271.

\bibitem[{Golub and Kahan(1965)}]{golkah:65}
Golub, G.~H., Kahan, W., 1965. Calculating the singular values and
  pseudo-inverse of a matrix. SIAM J. Numer. Anal. Ser. B 2, 205--224.

\bibitem[{Goreinov et~al.(2010)Goreinov, Oseledets, and Savostyanov}]{gos10}
Goreinov, S., Oseledets, I., Savostyanov, D., April 2010. Wedderburn rank
  reduction and {Krylov} subspace method for tensor approximation. {Part} 1:
  {Tucker} case. Tech. Rep. Preprint 2010-01, Inst. Numer. Math., Russian
  Academy of Sciences.

\bibitem[{Harshman(1970)}]{harsh70}
Harshman, R.~A., 1970. Foundations of the {PARAFAC} procedure: Models and
  conditions for an "explanatory" multi-modal factor analysis. UCLA Working
  Papers in Phonetics 16, 1--84.

\bibitem[{Hitchcock(1927)}]{hit:27}
Hitchcock, F.~L., 1927. Multiple invariants and generalized rank of a p-way
  matrix or tensor. J. Math. Phys. Camb. 7, 39--70.

\bibitem[{Ishteva et~al.(2009)Ishteva, De~Lathauwer, Absil, and
  Van~Huffel}]{ishteva09b}
Ishteva, M., De~Lathauwer, L., Absil, P.-A., Van~Huffel, S., 2009. Best low
  multilinear rank approximation of higher-order tensors, based on the
  {Riemannian} trust-region scheme. Tech. Rep. 09-142, ESAT-SISTA, K.U.Leuven
  (Leuven, Belgium).

\bibitem[{Jessup and Martin(2001)}]{jema01}
Jessup, E.~R., Martin, J.~H., 2001. Taking a new look at the latent semantic
  analysis approach to information retrieval. In: Berry, M.~W. (Ed.),
  Computational Information Retrieval. SIAM, Philadelphia, PA, pp. 121--144.

\bibitem[{Khoromskij and Khoromskaia(2009)}]{khor09}
Khoromskij, B.~N., Khoromskaia, V., 2009. Multigrid accelerated tensor
  approximation of function related multidimensional arrays. SIAM Journal on
  Scientific Computing 31~(4), 3002--3026.
\newline\urlprefix\url{http://link.aip.org/link/?SCE/31/3002/1}

\bibitem[{Kobayashi and Nomizu(1963)}]{kono:63}
Kobayashi, S., Nomizu, K., 1963. Foundations of Differential Geometry.
  Interscience Publisher.

\bibitem[{Kolda et~al.(2005)Kolda, Bader, and Kenny}]{kolda05}
Kolda, T.~G., Bader, .~W., Kenny, J.~P., November 2005. Higher-order web link
  analysis using multilinear algebra. In: ICDM 2005: Proceedings of the 5th
  IEEE International Conference on Data Mining. pp. 242--249.

\bibitem[{Kolda and Bader(2006)}]{kolda06b}
Kolda, T.~G., Bader, B.~W., 2006. The {TOPHITS} model for higher-order web link
  analysis. In: Proceedings of the SIAM Data Mining Conference Workshop on Link
  Analysis, Counterterrorism and Security.
\newline\urlprefix\url{http://www.siam.org/meetings/sdm06/workproceed/Link%20A%
nalysis/21Tamara_Kolda_SIAMLACS.pdf}

\bibitem[{Kolda and Bader(2009)}]{kolda09}
Kolda, T.~G., Bader, B.~W., 2009. Tensor decompositions and applications. SIAM
  Review 51~(3), 455--500.
\newline\urlprefix\url{http://link.aip.org/link/?SIR/51/455/1}

\bibitem[{Lehoucq et~al.(1998)Lehoucq, Sorensen, and Yang}]{arpack98}
Lehoucq, R., Sorensen, D., Yang, C., 1998. Arpack Users' Guide: Solution of
  Large Scale Eigenvalue Problems with Implicitly Restarted Arnoldi Methods.
  SIAM, Philadelphia.

\bibitem[{Lohr(2009, September, 21)}]{lohr09}
Lohr, S., 2009, September, 21. A \$1 million research bargain for netflix, and
  maybe a model for others. The New York Times.
\newline\urlprefix\url{http://www.nytimes.com/2009/09/22/technology/internet/2%
2netflix.html?_r=1}

\bibitem[{Oseledets et~al.(2008)Oseledets, Savostianov, and
  Tyrtyshnikov}]{oseledets08}
Oseledets, I.~V., Savostianov, D.~V., Tyrtyshnikov, E.~E., 2008. Tucker
  dimensionality reduction of three-dimensional arrays in linear time. SIAM
  Journal on Matrix Analysis and Applications 30~(3), 939--956.
\newline\urlprefix\url{http://link.aip.org/link/?SML/30/939/1}

\bibitem[{Savas(2003)}]{savas03}
Savas, B., 2003. Analyses and tests of handwritten digit recognition
  algorithms. Master's thesis, Linköping University,
  \texttt{http://www.mai.liu.se/\~{}besav/}.

\bibitem[{Savas and Eldén(2007)}]{savas07}
Savas, B., Eldén, L., 2007. Handwritten digit classification using higher order
  singular value decomposition. Pattern Recognition 40, 993--1003.

\bibitem[{Savas and Lim(2010)}]{savaslim10}
Savas, B., Lim, L.-H., 2010. {Q}uasi-{N}ewton methods on {G}rassmannians and
  multilinear approximations of tensors. Submitted to SIAM Journal on
  Scientific Computing.

\bibitem[{Stewart(2001)}]{stew:01}
Stewart, G.~W., 2001. Matrix Algorithms II: Eigensystems. SIAM, Philadelphia.

\bibitem[{Stewart(2002)}]{stew:02}
Stewart, G.~W., 2002. A {Krylov--Schur} algorithm for large eigenproblems. SIAM
  Journal on Matrix Analysis and Applications 23~(3), 601--614.
\newline\urlprefix\url{http://link.aip.org/link/?SML/23/601/1}

\bibitem[{Traud et~al.(2008)Traud, Kelsic, Mucha, and Porter}]{tkmp08}
Traud, A.~L., Kelsic, E.~D., Mucha, P.~J., Porter, M.~A., 2008. Community
  structure in online collegiate social networks. Tech. rep.,
  arXiv:physics.soc-ph/0809.0690.

\bibitem[{Tucker(1964)}]{tucke64}
Tucker, L.~R., 1964. The extension of factor analysis to three-dimensional
  matrices. Contributions to Mathematical Psychology, 109--127.

\bibitem[{Vasilescu and Terzopoulos(2002)}]{eccv:02}
Vasilescu, M. A.~O., Terzopoulos, D., 2002. Multilinear analysis of image
  ensembles: Tensorfaces. In: Proc. 7th European Conference on Computer Vision
  (ECCV'02). Lecture Notes in Computer Science, Vol. 2350. Springer Verlag,
  Copenhagen, Denmark, pp. 447--460.

\end{thebibliography}





% \end{thebibliography}

\end{document}